\documentclass[preprint,12pt]{elsarticle}

\usepackage{natbib}

\usepackage{graphicx}%
\usepackage{multirow}%
\usepackage{amsmath,amssymb,amsfonts}%
\usepackage{amsthm}%
\usepackage{mathrsfs}%
\usepackage[title]{appendix}%
\usepackage{xcolor}%
\usepackage{textcomp}%
\usepackage{manyfoot}%
\usepackage{booktabs}%
\usepackage{algorithm}%
\usepackage{algorithmicx}%
\usepackage{algpseudocode}%
\usepackage{listings}%
\usepackage{caption}
\usepackage[utf8]{inputenc}
\usepackage{comment}

\journal{Journal of Combinatorial Optimization}

\begin{document}

\newcommand{\Z}{{\mathbb Z}}
\newcommand{\N}{{\mathbb N}}
\newcommand{\Hi}{{\mathbb H}}
\newcommand{\R}{{\mathbb R}}
\newcommand{\Q}{{\mathbb Q}}
\newcommand{\ICG}{\mathrm{ICG}}
\newcommand{\WCG}{\mathrm{WCG}}
\newcommand{\WICG}{\mathrm{WICG}}

\newtheorem{theorem}{\bf Theorem}[section]
\newtheorem{corollary}[theorem]{\bf Corollary}
\newtheorem{lemma}[theorem]{\bf Lemma}
\newtheorem{proposition}[theorem]{\bf Proposition}
\newtheorem{conjecture}[theorem]{\bf Conjecture}
\newtheorem{remark}[theorem]{\bf Remark}
\newtheorem{problem}[theorem]{\bf Problem}
\newtheorem{definition}[theorem]{\bf Definition}

\newcommand{\QED} {\hfill$\square$}

%\newenvironment{proof} {\par \noindent \textbf{Proof. }}{\QED \par \bigskip \par}
% Centering the EPS graphics inside the FIGURE environment
\def\slika #1{\begin{center} \epsffile{#1} \end{center}}

\begin{frontmatter}

\title{Trees with maximum $\sigma$-irregularity under a prescribed maximum degree 6}

\author[address1]{Milan Ba\v si\'c\corref{mycorrespondingauthor}}
\cortext[mycorrespondingauthor]{Corresponding author}
\ead{basic\_milan@yahoo.com}
\address[address1]{University of Ni\v s, Ni\v s, Serbia}

\begin{abstract}
The sigma-irregularity index $\sigma(G) = \sum_{uv \in E(G)} (d_G(u) - d_G(v))^2$ measures the total degree imbalance along the edges of a graph.
We study extremal problems for $\sigma(T)$ within the class of trees of fixed order $n$ and bounded maximum degree $\Delta = 6$.
Using a penalty-function framework combined with handshake identities and congruence arguments, we determine the exact maximum value of $\sigma(T)$ for every residue class of $n$ modulo $6$, showing that the possible minimum values of the penalty function are $0, 10, 20, 22, 30,$ and $40$.
For each case, we provide a complete characterization of all maximizing trees in terms of degree counts and edge multiplicities. In five of the six residue classes, all extremal trees contain only vertices of degrees $1, 2,$ and $6$, while for $n \equiv 3 \pmod{6}$ an additional exceptional family arises involving vertices of degree $3$. These results extend earlier work on sigma-irregularity for smaller degree bounds and illustrate the rapidly growing combinatorial complexity of the problem as the maximum degree increases.
\end{abstract}


\end{frontmatter}

%\linenumbers


\section{Introduction}

A graph is called \emph{irregular} if not all vertices have the same degree. 
Quantifying how far a graph is from regularity has long been of interest in both 
theoretical and applied graph theory. One of the earliest measures of degree 
heterogeneity is the \emph{degree variance} proposed by Snijders~\cite{Snijders1981}. 
Later, Albertson~\cite{Albertson1997} introduced a prominent irregularity index
\[
\mathrm{irr}(G) \;=\; \sum_{uv\in E(G)} \bigl| d(u)-d(v) \bigr|,
\]
which is now commonly called the \emph{Albertson irregularity}. 
This invariant has been extensively studied; see, for example, 
\cite{Albertson1997,Abdo2014,Hansen2005} and the references therein.

Degree–based irregularity measures play an important role in several application 
areas, including network science and mathematical chemistry. 
They have been used to quantify structural heterogeneity in communication and 
social networks~\cite{Estrada2010}, and as molecular descriptors in 
quantitative structure–property relationships (QSPR) for chemical 
compounds~\cite{Reti2018QSPR,Gutman2005}. 
These applications have stimulated the search for alternative irregularity 
indices that are more sensitive to large local degree differences.

Among such variants, a particularly useful one is the so-called 
\emph{$\sigma$-irregularity index}, or simply the \emph{$\sigma$-index}, defined 
for any graph $G=(V,E)$ by
\[
\sigma(G) \;=\; \sum_{uv\in E(G)} \bigl(d(u)-d(v)\bigr)^2 .
\]
This index can be viewed as a quadratic analog of Albertson’s measure. 
By squaring degree differences, $\sigma(G)$  emphasizes large local disparities 
between adjacent vertices, making it a sensitive indicator of structural 
irregularity. The $\sigma$-index was systematically studied in 
\cite{Abdo2018}, where extremal graphs with maximum $\sigma$-irregularity for a 
given order were characterized. 
Further fundamental properties were developed in \cite{Gutman2018}, including a 
solution of the inverse $\sigma$-index problem. 
Additional investigations and comparisons with other irregularity measures 
appear in \cite{Reti2019}. Arif et al.~\cite{Arif2023} studied irregularity indices, including the $\sigma$-index, for graph families characterized by two main eigenvalues, and applied the resulting values in QSPR analysis.

A natural and active direction of research concerns extremal problems for 
$\sigma(G)$ under various structural constraints. 
For general connected graphs of fixed order, Abdo, Dimitrov and Gutman~\cite{Abdo2018} 
determined the unique graph that maximizes $\sigma(G)$. 
It is therefore natural to ask how this problem behaves when the class of admissible 
graphs is restricted. Among the most important such restrictions is the class of trees, 
which is both mathematically fundamental and highly relevant in applications.

In the chemical context, trees arise as molecular graphs without cycles; in this setting, 
one usually speaks of \emph{chemical trees} when vertex degrees are bounded by $4$. 
The $\sigma$-irregularity of chemical trees was investigated in 
\cite{Kovijanic2024}, where a complete characterization of trees with 
maximum $\sigma$ among all chemical trees of fixed order was obtained.

More recently, Dimitrov \emph{et al.}~\cite{Dimitrov2026} initiated a systematic 
study of $\sigma$-extremal trees under a prescribed maximum degree $\Delta$. 
They established general structural properties of $\sigma$-maximal trees for 
arbitrary $\Delta\ge 3$, and then focused in detail on the case $\Delta=5$. 
In that setting, they showed that every $n$-vertex tree maximizing $\sigma(G)$ 
with $\Delta(G)=5$ contains only vertices of degrees $1$, $2$, and $5$. 
Moreover, almost all edges in such extremal trees join a leaf to a degree-$5$ 
vertex or a degree-$5$ vertex to a degree-$2$ vertex, while edges between two 
degree-$5$ vertices or two degree-$2$ vertices are extremely rare. 
These findings support a general heuristic: for fixed maximum degree $\Delta$, a tree tends to exhibit larger $\sigma$-irregularity when its degree distribution is strongly polarized, with many vertices of degree $\Delta$ and many leaves, while vertices of degree $2$ occur in comparable number to the $\Delta$-vertices and serve primarily as connectors between the two extremes.

In the present paper, we advance this program to the next case $\Delta=6$ and 
determine the trees with maximum $\sigma$-irregularity among all $n$-vertex trees 
subject to $\Delta(G)=6$. 
Our results show that the structural paradigm observed for $\Delta=5$ persists, while exhibiting an additional structural feature. 
For most values of $n$, the $\sigma$-maximal trees with $\Delta=6$ have the same 
degree set $\{1,2,6\}$ as in the $\Delta=5$ case, consisting of degree-$6$ vertices 
connected through degree-$2$ vertices to many leaves. 
However, when the number of vertices satisfies $n\equiv 3 \pmod{6}$, an 
additional extremal configuration appears. 
In this exceptional case, the $\sigma$-maximal tree contains a vertex of degree 
$3$, so that the degree set becomes $\{1,2,3,6\}$. 
To the best of our knowledge, this is the first instance where a vertex of degree 
other than $1$, $2$, or $\Delta$ is needed to construct a $\sigma$-maximal tree 
under a prescribed maximum degree constraint.

The paper is organized as follows. 
In Section~2 we introduce notation and derive a penalty decomposition of 
$\sigma(G)$ for trees with $\Delta=6$, which reduces the extremal problem to a 
 optimization problem over edge–degree multiplicities. 
In Section~3 we develop a sequence of structural lemmas that exclude vertices of 
degrees $3$, $4$, and $5$ in most residue classes of $n$ modulo $6$, with a single 
exceptional case. 
In Section~4 we state and prove six main theorems that completely characterize 
$\sigma$-maximal trees with $\Delta=6$ for each congruence class of $n$ modulo $6$, 
and we provide explicit constructions for all extremal families. 
Finally, Section~5 contains concluding remarks and directions for future 
research, including possible extensions to larger maximum degrees.

%========================================================
\section{Preliminaries and notation}
\label{sec:prelim}

Let $d_G(u)$ denote the degree of a vertex $u$ in a graph $G$.
We begin by introducing the main invariant studied in this paper.

\begin{definition}
\label{def:sigma}
For a connected graph $G$ with edge set $E(G)$, the \emph{sigma-regularity index}
is defined by
\begin{equation}
\label{eq:def_sigma}
\sigma(G)\;=\;\sum_{uv\in E(G)}\bigl(d_G(u)-d_G(v)\bigr)^2.
\end{equation}
\end{definition}

The expression for $\sigma(G)$ can be restated as follows. For $1\le i\le j\le \Delta$,
let $m_{i,j}$ denote the number of edges joining a vertex of degree $i$ to a vertex
of degree $j$. Then
\[
\sigma(G)
=\sum_{1\le i\le j\le \Delta} m_{i,j}\,(i-j)^2
=\sum_{1\le i\le j\le \Delta}\sigma_{i,j}\,m_{i,j},
\]
where we write
\[
\sigma_{i,j}=(i-j)^2.
\]
This provides a shorter and more convenient notation.

\medskip
Throughout the paper we work with trees of maximum degree $\Delta=6$.
Let $T$ be a tree on $n$ vertices with maximum degree $\Delta=6$.
For $1\le i\le 6$, denote by $n_i$ the number of vertices of degree $i$, and for
$1\le i\le j\le 6$, denote by $m_{i,j}$ the number of edges joining a vertex
of degree $i$ to a vertex of degree $j$.

\medskip
The variables $n_i$ and $m_{i,j}$ satisfy the following identities, valid for every
tree $T$ with maximum degree $\Delta=6$:
\begin{align}
n_1+n_2+\cdots+n_6&=n, \label{eq:tree-sum-ni}\\
n_1+2n_2+\cdots+6n_6&=2n-2, \label{eq:tree-sum-deg}\\
m_{12}+m_{13}+m_{14}+m_{15}+m_{16}&=n_1, \label{eq:tree-ni-1}\\
m_{12}+2m_{22}+m_{23}+m_{24}+m_{25}+m_{26}&=2n_2, \label{eq:tree-ni-2}\\
m_{13}+m_{23}+2m_{33}+m_{34}+m_{35}+m_{36}&=3n_3, \label{eq:tree-ni-3}\\
m_{14}+m_{24}+m_{34}+2m_{44}+m_{45}+m_{46}&=4n_4, \label{eq:tree-ni-4}\\
m_{15}+m_{25}+m_{35}+m_{45}+2m_{55}+m_{56}&=5n_5, \label{eq:tree-ni-5}\\
m_{16}+m_{26}+m_{36}+m_{46}+m_{56}+2m_{66}&=6n_6. \label{eq:tree-ni-6}
\end{align}

\medskip
We define
\[
P(T)=\sum_{(i,j)\in S} F(i,j)\,m_{i,j},
\qquad
S=\{(i,j):1\le i\le j\le 6,\ (i,j)\notin\{(1,6),(2,6)\}\},
\]
where the coefficients $F(i,j)$ are given by
\[
\begin{aligned}
F(1,2) &= -\sigma_{1,2} + \tfrac{5}{3}\sigma_{1,6} - \tfrac{2}{3}\sigma_{2,6},\\
F(1,3) &= -\sigma_{1,3} + \tfrac{4}{3}\sigma_{1,6} - \tfrac{1}{3}\sigma_{2,6},\\
F(1,4) &= -\sigma_{1,4} + \tfrac{7}{6}\sigma_{1,6} - \tfrac{1}{6}\sigma_{2,6},\\
F(1,5) &= -\sigma_{1,5} + \tfrac{16}{15}\sigma_{1,6} - \tfrac{1}{15}\sigma_{2,6},\\[2mm]
F(2,2) &= \tfrac{2}{3}\sigma_{1,6} - \sigma_{2,2} + \tfrac{1}{3}\sigma_{2,6},\\
F(2,3) &= \tfrac{1}{3}\sigma_{1,6} - \sigma_{2,3} + \tfrac{2}{3}\sigma_{2,6},\\
F(2,4) &= \tfrac{1}{6}\sigma_{1,6} - \sigma_{2,4} + \tfrac{5}{6}\sigma_{2,6},\\
F(2,5) &= \tfrac{1}{15}\sigma_{1,6} - \sigma_{2,5} + \tfrac{14}{15}\sigma_{2,6},\\[2mm]
F(3,3) &= \sigma_{2,6} - \sigma_{3,3},\\
F(3,4) &= -\tfrac{1}{6}\sigma_{1,6} + \tfrac{7}{6}\sigma_{2,6} - \sigma_{3,4},\\
F(3,5) &= -\tfrac{4}{15}\sigma_{1,6} + \tfrac{19}{15}\sigma_{2,6} - \sigma_{3,5},\\
F(3,6) &= -\tfrac{1}{3}\sigma_{1,6} + \tfrac{4}{3}\sigma_{2,6} - \sigma_{3,6},\\[2mm]
F(4,4) &= -\tfrac{1}{3}\sigma_{1,6} + \tfrac{4}{3}\sigma_{2,6} - \sigma_{4,4},\\
F(4,5) &= -\tfrac{13}{30}\sigma_{1,6} + \tfrac{43}{30}\sigma_{2,6} - \sigma_{4,5},\\
F(4,6) &= -\tfrac{1}{2}\sigma_{1,6} + \tfrac{3}{2}\sigma_{2,6} - \sigma_{4,6},\\[2mm]
F(5,5) &= -\tfrac{8}{15}\sigma_{1,6} + \tfrac{23}{15}\sigma_{2,6} - \sigma_{5,5},\\
F(5,6) &= -\tfrac{3}{5}\sigma_{1,6} + \tfrac{8}{5}\sigma_{2,6} - \sigma_{5,6},\\
F(6,6) &= -\tfrac{2}{3}\sigma_{1,6} + \tfrac{5}{3}\sigma_{2,6} - \sigma_{6,6}.
\end{aligned}
\]

Table~\ref{tab:F-delta6-row} lists all values of $F(i,j)$ for $\Delta=6$, sorted
in increasing order.

\begin{table}[ht]
\centering
\caption{Values of the penalty function $F(i,j)$ for $\Delta=6$, arranged row-wise
in increasing order for $(i,j)\in S=\{(i,j):1\le i\le j\le 6,\ (i,j)\notin\{(1,6),(2,6)\}\}$.}
\label{tab:F-delta6-row}
\begin{tabular}{cccc}
\hline
$(i,j)$ & $F(i,j)$ & $(i,j)$ & $F(i,j)$ \\ \hline
$(3,6)$ & $4.0$   & $(4,6)$ & $7.5$   \\
$(2,5)$ & $7.6$   & $(1,5)$ & $9.6$   \\
$(3,5)$ & $9.6$   & $(5,6)$ & $9.6$   \\
$(6,6)$ & $10.0$  & $(4,5)$ & $11.1$  \\
$(5,5)$ & $11.2$  & $(4,4)$ & $13.0$  \\
$(3,4)$ & $13.5$  & $(2,4)$ & $13.5$  \\
$(3,3)$ & $16.0$  & $(1,4)$ & $17.5$  \\
$(2,3)$ & $18.0$  & $(2,2)$ & $22.0$  \\
$(1,3)$ & $24.0$  & $(1,2)$ & $30.0$  \\
$(1,1)$ & $40.0$  &         &         \\ \hline
\end{tabular}
\end{table}

%========================================================
\subsection{Decomposition of $\sigma(T)$ and the reduced optimization problem}
\label{sec:decomposition}

In this section we eliminate the variables $m_{1,6}$ and $m_{2,6}$ from the handshake
system \eqref{eq:tree-sum-ni}--\eqref{eq:tree-ni-6} and express them explicitly in terms
of the remaining edge-multiplicities $m_{i,j}$ with $(i,j)\in S$ and the order $n$.

\medskip
Solving the linear system obtained from
\eqref{eq:tree-sum-ni} -- \eqref{eq:tree-ni-6},
we obtain
\begin{align}
m_{1,6}
&=
\frac{2n}{3}+\frac{4}{3}
-\frac{5}{3}m_{1,2}
-\frac{4}{3}m_{1,3}
-\frac{7}{6}m_{1,4}
-\frac{16}{15}m_{1,5}
-\frac{2}{3}m_{2,2}
-\frac{1}{3}m_{2,3}
\nonumber\\
&\hspace{1.2cm}
-\frac{1}{6}m_{2,4}
-\frac{1}{15}m_{2,5}
+\frac{1}{6}m_{3,4}
+\frac{4}{15}m_{3,5}
+\frac{1}{3}m_{3,6}
+\frac{1}{3}m_{4,4}
+\frac{13}{30}m_{4,5}
\nonumber\\
&\hspace{1.2cm}
+\frac{1}{2}m_{4,6}
+\frac{8}{15}m_{5,5}
+\frac{3}{5}m_{5,6}
+\frac{2}{3}m_{6,6},
\label{eq:m16-general}
\\[2mm]
m_{2,6}
&=
\frac{n}{3}-\frac{7}{3}
+\frac{2}{3}m_{1,2}
+\frac{1}{3}m_{1,3}
+\frac{1}{6}m_{1,4}
+\frac{1}{15}m_{1,5}
-\frac{1}{3}m_{2,2}
-\frac{2}{3}m_{2,3}
\nonumber\\
&\hspace{1.2cm}
-\frac{5}{6}m_{2,4}
-\frac{14}{15}m_{2,5}
-m_{3,3}
-\frac{7}{6}m_{3,4}
-\frac{19}{15}m_{3,5}
-\frac{4}{3}m_{3,6}
-\frac{4}{3}m_{4,4}
\nonumber\\
&\hspace{1.2cm}
-\frac{43}{30}m_{4,5}
-\frac{3}{2}m_{4,6}
-\frac{23}{15}m_{5,5}
-\frac{8}{5}m_{5,6}
-\frac{5}{3}m_{6,6}.
\label{eq:m26-general}
\end{align}

These identities are used in the proof of Theorem~\ref{thm:delta6-P0} to decompose
$\sigma(T)$ into a constant term depending only on $n$ and the penalty term $P(T)$. Indeed, substituting \eqref{eq:m16-general}–\eqref{eq:m26-general} into
$\Sigma(T)=\sum_{1\le i\le j\le 6}\sigma_{i,j}m_{i,j}$, where
$\sigma_{i,j}=(i-j)^2$,
and then collecting like terms, we can obtain
\[
\Sigma(T)=C(n)-P(T),
\]
with
\[
P(T)=\sum_{(i,j)\in S}F(i,j)\,m_{i,j}.
\]

\medskip
Throughout this paper, let $\mathcal{T}_n^{(6)}$ be the family of all trees on $n$
vertices with maximum degree $\Delta=6$. For $n\ge 1$, define
\[
P_{\min}(n):=\min\{P(T):\, T\in \mathcal{T}_n^{(6)}\}.
\]
Since $\sigma(T)$ differs from a constant (depending only on $n$) by $-P(T)$, maximizing
$\sigma(T)$ over $\mathcal{T}_n^{(6)}$ is equivalent to minimizing $P(T)$ over
$\mathcal{T}_n^{(6)}$.

%========================================================
\section{Structural exclusions in the extremal range}
\label{sec:exclusions}

In order to determine all minimizers of $P(T)$, we first show that for the penalty
ranges relevant to extremality, vertices of degrees $4$ and $5$ do not occur, and
vertices of degree $3$ occur only in one explicit exceptional configuration.
This reduces the subsequent case analysis to a small collection of edge types.

\subsection{Excluding degrees $4$ and $5$}
\label{subsec:no45}

\begin{lemma}
\label{lem:no5-threshold}
Let $\Delta=6$ and let $T$ be a tree.
\begin{itemize}
\item If $P(T)\le 30$, then $n_5=0$.
\item If $n\not\equiv 0\pmod 6$ and $P(T)\le 40$, then $n_5=0$.
\end{itemize}
\end{lemma}

\begin{proof}
Assume that $n_5\ge 1$.
From the degree--$5$ handshake equation \eqref{eq:tree-ni-5} we have
\[
m_{1,5}+m_{2,5}+m_{3,5}+m_{4,5}+2m_{5,5}+m_{5,6}=5n_5.
\]
%Thus there are exactly $5n_5$ incidences at degree--$5$ vertices.

We distinguish two cases.

\medskip
\noindent\textbf{Case 1: $n_5=1$.}
Then there is exactly one vertex of degree $5$, hence there is no edge of type $(5,5)$ and $m_{5,5}=0$.

Consequently,
\[
m_{1,5}+m_{2,5}+m_{3,5}+m_{4,5}+m_{5,6}=5.
\]
By Table~\ref{tab:F-delta6-row}, among the edge types incident to degree $5$ with the other endpoint
different from $5$, the smallest value is $F(2,5)=7.6$. Therefore,
\[
P(T)\ \ge\ 7.6\bigl(m_{1,5}+m_{2,5}+m_{3,5}+m_{4,5}+m_{5,6}\bigr)=7.6\cdot 5=38.
\]
In particular, $P(T)\le 30$ is impossible, and if $P(T)\le 40$ then necessarily $P(T)=38$.
Moreover, since $38<40$ and every positive value $F(i,j)$ is at least $4$,
the equality $P(T)=38$ forces
\[
m_{2,5}=5
\qquad\text{and}\qquad
m_{i,j}=0\ \text{for all }(i,j)\in S\setminus\{(2,5)\}.
\]
Now the degree--$6$ handshake row \eqref{eq:tree-ni-6} reduces to
\[
m_{1,6}+m_{2,6}=6n_6,
\]
so $m_{1,6}+m_{2,6}\equiv 0\pmod 6$.
On the other hand, counting edges gives
\[
m_{1,6}+m_{2,6}= (n-1)-m_{2,5} = (n-1)-5 = n-6.
\]
Hence $n\equiv 0\pmod 6$, contradicting the assumption $n\not\equiv 0\pmod 6$ in the second claim.

\medskip
\noindent\textbf{Case 2: $n_5\ge 2$.}
From the degree--$5$ handshake equation \eqref{eq:tree-ni-5} we have $2m_{5,5}\le 5n_5$ and hence $m_{5,5}\le \frac{5n_5}{2}$.

By Table~\ref{tab:F-delta6-row}, among all edge types incident to degree $5$, the minimum value of $F$  occurs at $(2,5)$ with $F(2,5)=7.6$, while $F(5,5)=11.2$.

Therefore,
\[
P(T)
\ \ge\ 
11.2\,m_{5,5}
+7.6\bigl(5n_5-2m_{5,5}\bigr)
\ =\ 
7.6\cdot 5n_5 -4.0\,m_{5,5}.
\]
Using $m_{5,5}\le \frac{5n_5}{2}$ yields
\[
P(T)
\ \ge\ 
7.6\cdot 5n_5
-4.0\cdot \frac{5n_5}{2}
\ =\ 
28n_5
\ \ge\ 
56,
\]
which contradicts both bounds $P(T)\le 30$ and $P(T)\le 40$.

\end{proof}

We next turn to vertices of degree $4$. Since the smallest penalties associated with
edges incident to degree $4$ vertices are already relatively large
(cf.\ Table~\ref{tab:F-delta6-row}), their presence quickly forces $P(T)$ above the
extremal thresholds considered here. This observation leads to the following
exclusion result.

\begin{lemma}
\label{lem:no4-threshold}
Let $\Delta=6$ and let $T$ be a tree on $n$ vertices.

\begin{itemize}
\item If $P(T)<30$, then $n_4=0$.
\item If $n\not\equiv 1 \pmod 6$ and $P(T)=30$, then $n_4=0$.
\item If $n\not\equiv \{0,1,2\} \pmod 6$ and $30<P(T)\le 40$, then $n_4=0$.
\end{itemize}
\end{lemma}

\begin{proof}
Assume $n_4\ge 1$.

From the degree--$4$ handshake equation we have
\[
m_{1,4}+m_{2,4}+m_{3,4}+m_{4,5}+m_{4,6}+2m_{4,4}=4n_4 \ge 4 .
\]
Hence there are at least $4$ incidences at degree--$4$ vertices. 
Indeed, if $m_{4,4}=0$, this is immediate from $4n_4\ge 4$, while if
$m_{4,4}\ge 1$, then necessarily $n_4\ge 2$ and hence $4n_4\ge 8$.
By Table~\ref{tab:F-delta6-row}, the smallest penalty among all edge--types
incident to degree $4$ is $F(4,6)=7.5$, so
\[
P(T)\ \ge\ 4\cdot 7.5 \ =\ 30 .
\]
This proves the first item.

\medskip
\noindent
\textbf{Case 1: $P(T)=30$ and $n_4\ge 1$.}  
Equality in the above bound forces all incidences at degree--$4$ vertices
to come from $(4,6)$--edges. Hence
\[
n_4=1, \qquad m_{4,6}=4, \qquad m_{i,j}=0 \ \text{for all }(i,j)\in S\setminus\{(4,6)\}.
\]
Let $x=m_{1,6}$ and $y=m_{2,6}$.  
Counting edges gives
\[
x+y+4=n-1 ,
\]
while the degree--$6$ handshake equation (\ref{eq:tree-ni-6}) yields
\[
x+y+4=6n_6 .
\]
Thus $6n_6=n-1$, i.e.\ $n\equiv 1\pmod 6$.  
Therefore, if $n\not\equiv 1\pmod 6$, then $P(T)=30$ is impossible with $n_4\ge 1$,
which proves the second item.

\medskip
\noindent
\textbf{Case 2: $30<P(T)\le 40$ and $n\not\equiv\{0,1,2\}\pmod 6$.}  
Assume first that $n_4\ge 2$. Then
\[
m_{1,4}+m_{2,4}+m_{3,4}+m_{4,5}+m_{4,6}+2m_{4,4}=4n_4 \ge 8 .
\]

If $m_{4,4}=0$, then all incidences come from edges incident to degree $4$
of types $(1,4)$, $(2,4)$, $(3,4)$, $(4,5)$ or $(4,6)$.  
The smallest penalty among these is $F(4,6)=7.5$, hence
\[
P(T)\ \ge\ 8\cdot 7.5 \ =\ 60 \ >\ 40 ,
\]
a contradiction.

If $m_{4,4}\ge 1$, then one $(4,4)$--edge contributes two incidences at cost $13$,
and the remaining at least six incidences contribute at least $6\cdot 7.5=45$.
Hence
\[
P(T)\ \ge\ 13 + 45 \ =\ 58 \ >\ 40 ,
\]
again a contradiction.

Therefore $n_4\ge 2$ is impossible, and we must have $n_4=1$.

By Lemma~\ref{lem:no5-threshold} (applied under $P(T)\le 40$ and
$n\not\equiv 0\pmod 6$), we also have $n_5=0$, hence $m_{4,5}=0$ and $m_{4,4}=0$.
Thus the degree--$4$ handshake reduces to
\[
m_{1,4}+m_{2,4}+m_{3,4}+m_{4,6}=4 .
\]

\medskip
\noindent\textbf{Case 2.1: $m_{4,6}\le 2$.}
Then $m_{1,4}+m_{2,4}+m_{3,4}\ge 2$, so at least two incidences come from
$(1,4)$, $(2,4)$, or $(3,4)$.
By Table~\ref{tab:F-delta6-row}, each such incidence has penalty at least $13.5$.
Hence
\[
P(T)\ge 2\cdot 7.5+2\cdot 13.5=42>40,
\]
a contradiction. Therefore $m_{4,6}\in\{3,4\}$.

\medskip
\noindent
\textbf{Case 2.2: $m_{4,6}=4$.}  
Then $m_{1,4}=m_{2,4}=m_{3,4}=0$, and the remaining penalty budget is at most
$P(T)-30\le 10$.  
Since $F(3,6)=4$ and $F(6,6)=10$ are the only values in $S$ not exceeding $10$
(with $n_5=0$), we must have
\[
m_{i,j}=0 \ \text{for all }(i,j)\in S\setminus\{(3,6),(4,6),(6,6)\}, \qquad
m_{3,6}\le 2, \quad m_{6,6}\le 1 .
\]

Let $x=m_{1,6}$, $y=m_{2,6}$, $t=m_{3,6}$ and $b=m_{6,6}$.  
Counting edges gives
\begin{eqnarray}
\label{edge_count_4}
  x+y+t+b+4=n-1 ,  
\end{eqnarray}

and the degree--$6$ handshake equation yields
\begin{eqnarray}
\label{handshake_6_4}
x+y+t+4+2b=6n_6 .
\end{eqnarray}
Subtracting (\ref{edge_count_4}) from (\ref{handshake_6_4}) gives
\[
b=6n_6-(n-1) .
\]
Since $n\not\equiv\{0,1\}\pmod 6$, we obtain $b\ge 2$, contradicting $b\le 1$.
Hence this case is impossible.

\medskip
\noindent
\textbf{Case 2.3: $m_{4,6}=3$.}  
Then
\[
m_{1,4}+m_{2,4}+m_{3,4}=1 .
\]

\smallskip
\noindent
\textbf{Case 2.3(a): the unique non-$(4,6)$ edge is $(1,4)$.}  
Then
\[
P(T)\ \ge\ 3\cdot 7.5 + 17.5 \ =\ 40 ,
\]
so $P(T)=40$ and no further $(i,j)\in S$ may occur.  
Hence $m_{6,6}=m_{3,6}=0$.

Let $x=m_{1,6}$ and $y=m_{2,6}$.  
Edge counting gives
\[
x+y+3+1=n-1 \quad\Rightarrow\quad x+y=n-5 ,
\]
while the degree--$6$ handshake gives
\[
x+y+3=6n_6 .
\]
Thus $6n_6=n-2$, impossible for $n\not\equiv 2\pmod 6$.

\smallskip
\noindent
\textbf{Case 2.3(b): the unique non-$(4,6)$ edge is $(2,4)$ or $(3,4)$.}  
Then
\[
P(T)\ \ge\ 3\cdot 7.5 + 13.5 \ =\ 36 ,
\]
so the remaining budget is at most $4$.  
Since $F(3,6)=4$ is the only admissible positive value in $S$ not exceeding $4$,
we must have
\[
m_{i,j}=0 \ \text{for all }(i,j)\in S\setminus\{(3,6),(4,6),(2,4),(3,4)\},
\qquad m_{3,6}\le 1 .
\]

The degree--$3$ handshake equation becomes
\[
m_{3,4}+m_{3,6}=3n_3 .
\]
Since $m_{3,4}\in\{0,1\}$ and $m_{3,6}\le 1$, this forces
\[
m_{3,4}=m_{3,6}=0, \qquad n_3=0 .
\]

Now edge counting gives
\[
x+y+3+1=n-1 \quad\Rightarrow\quad x+y=n-5 ,
\]
and the degree--$6$ handshake yields
\[
x+y+3=6n_6 .
\]
Thus again $6n_6=n-2$, impossible for $n\not\equiv 2\pmod 6$.

\medskip
Since all cases lead to contradictions, we conclude that $n_4=0$ whenever
$n\not\equiv\{0,1,2\}\pmod 6$ and $30<P(T)\le 40$.  
This proves the third item.
\end{proof}

\subsection{Excluding degree $3$}
\label{subsec:no3}

The elimination of vertices of degree $3$ is more subtle. Unlike degrees $4$ and $5$,
degree $3$ can interact with degree $6$ at relatively low penalty cost via edges of
type $(3,6)$. As a result, its exclusion depends not only on the penalty threshold
but also on congruence conditions on $n$. We therefore treat the ranges $P(T)\le 10$,  $10<P(T)\le 22$,
$22<P(T)\le 30$ and $30<P(T)\le 40$ separately.

\medskip

\begin{lemma}
\label{lem:no3-threshold}
Let $\Delta=6$ and let $T$ be a tree on $n$ vertices.
\begin{itemize}
\item If $P(T)\le 10$, then $n_3=0$.
\item If $n\not\equiv \{0,1\} \pmod 6$ and $10<P(T)\le 22$, then $n_3=0$.
\end{itemize}
\end{lemma}

\begin{proof}
Assume to the contrary that $n_3\ge 1$.

In item~1, since $P(T)\le 10<30$, Lemma~\ref{lem:no4-threshold} (first item) gives
$n_4=0$, and Lemma~\ref{lem:no5-threshold} gives $n_5=0$.
In item~2, $P(T)\le 22<30$ again yields $n_4=0$ from
Lemma~\ref{lem:no4-threshold} (first item), and $n_5=0$ from
Lemma~\ref{lem:no5-threshold}.
Hence in both items,
\[
n_4=n_5=0.
\]

Since $F(1,3)=24>22$, we must have
\[
m_{1,3}=0.
\]
In item~1, $P(T)\le 10$ and $F(2,3)=18$ also imply $m_{2,3}=0$.
In item~2, $m_{2,3}\ge 2$ would give $P(T)\ge 36>22$, so only $m_{2,3}=1$ needs
to be excluded.

Using the degree--$3$ handshake together with $n_4=n_5=0$ and $m_{1,3}=0$, we obtain
\[
m_{2,3}+2m_{3,3}+m_{3,6}=3n_3.
\]
If $m_{2,3}=1$, then $2m_{3,3}+m_{3,6}\equiv 2\pmod 3$.
If $m_{3,3}\ge 1$, then $P(T)\ge 18+16=34>22$, impossible; hence $m_{3,3}=0$.
Thus $m_{3,6}\equiv 2\pmod 3$, so $m_{3,6}\ge 2$ and
\[
P(T)\ge 18+2\cdot 4=26>22,
\]
again impossible.
Therefore,
\[
m_{2,3}=0.
\]

With $n_4=n_5=0$, $m_{1,3}=0$, and $m_{2,3}=0$, the degree--$3$ handshake reduces to
\[
2m_{3,3}+m_{3,6}=3n_3.
\]
If $m_{3,3}\ge 1$, then $n_3\ge 2$, and the cheapest possibility
$m_{3,3}=1$, $m_{3,6}=4$ gives
\[
P(T)\ge 16+4\cdot 4=32>22,
\]
impossible. Hence
\[
m_{3,3}=0,
\qquad
m_{3,6}=3n_3\ge 3.
\]

Since $F(3,6)=4$, we have $P(T)\ge 4m_{3,6}$.
Thus:
\begin{itemize}
\item in item~1, $P(T)\ge 12>10$, a contradiction;
\item in item~2, $P(T)\le 22$ forces $m_{3,6}=3$ and hence $n_3=1$.
\end{itemize}

In item~2, the three $(3,6)$--edges already contribute $12$ to $P(T)$.
Since $F(2,2)=22$ and $F(1,2)=30$, any edge of type $(2,2)$ or $(1,2)$ would force
$P(T)>22$. Hence
\[
m_{2,2}=0,
\qquad
m_{1,2}=0.
\]
Also $P(T)\le 22$ implies $m_{6,6}\le 1$.

Let $x=m_{1,6}$, $y=m_{2,6}$, and $b=m_{6,6}\in\{0,1\}$.
Counting edges gives
\[
x+y+m_{3,6}+b=n-1,
\]
and the degree--$6$ handshake yields
\[
x+y+m_{3,6}+2b=6n_6.
\]
Subtracting these two equations gives
\[
b=6n_6-(n-1),
\qquad\text{so}\qquad
b\equiv -(n-1)\pmod 6.
\]
Since $b\in\{0,1\}$, this forces
\[
n\equiv 1 \ \text{or}\ 0 \pmod 6,
\]
contradicting the hypothesis in item~2.

We conclude that $n_3=0$ holds in both items, completing the proof.
\end{proof}

\begin{lemma}
\label{lem:no3-22-30}
Let $\Delta=6$ and let $T$ be a tree on $n$ vertices.
Assume that $n\not\equiv \{0,1,2\}\pmod 6$ and $22<P(T)\le 30$. 
Then $n_3=0$.
\end{lemma}

\begin{proof}
Assume to the contrary that $n_3\ge 1$.

Since $P(T)\le 30$, Lemma~\ref{lem:no5-threshold} (first item) yields $n_5=0$.
Moreover, as $n\not\equiv 1\pmod 6$ and $P(T)\le 30$,
Lemma~\ref{lem:no4-threshold} (first and second items) yields $n_4=0$.
Hence
\[
n_4=n_5=0.
\]

We now prove that $m_{1,3}=0$.
Assume $m_{1,3}\ge 1$.
Since $n_3\ge 1$, the degree--$3$ handshake equation forces at least three edges
having a degree--$3$ endpoint.
One of them is of type $(1,3)$, and the remaining two contribute at least
$2\cdot F(3,6)=8$, because $F(3,i)\ge F(3,6)=4$ for every $i$.
Therefore
\[
P(T)\ \ge\ F(1,3)+2\cdot 4\ =\ 24+8\ =\ 32\ >30,
\]
a contradiction. Hence $m_{1,3}=0$.

Next we show that $m_{2,3}=0$.
Assume for a contradiction that $m_{2,3}\ge 1$.

If $m_{3,3}\ge 1$, then
\[
P(T)\ge F(2,3)+F(3,3)=18+16=34>30,
\]
which is impossible.

Thus $m_{3,3}=0$.
Using the degree--$3$ handshake together with $n_4=n_5=0$ and $m_{1,3}=0$, we obtain
\[
m_{2,3}+m_{3,6}=3n_3.
\]
If $m_{2,3}\ge 2$, then $P(T)\ge 2F(2,3)=36>30$, impossible.
Hence $m_{2,3}=1$.
The above equation then yields $m_{3,6}\equiv 2\pmod 3$, so
$m_{3,6}\in\{2,5,8,\dots\}$.
Since $P(T)\le 30$, we must have $m_{3,6}=2$, and therefore
\[
P(T)\ \ge\ 18+2\cdot 4\ =\ 26.
\]

We now rule out the configuration
$m_{2,3}=1$, $m_{3,3}=0$, $m_{3,6}=2$.
Because $n_4=n_5=0$ and $m_{1,3}=m_{3,3}=0$, every $S$--edge other than $(2,3)$ and
$(3,6)$ must be of type $(2,2)$ or $(1,2)$.
Since $P(T)\le 30$ and the contribution above is already $26$, we have room for
at most $4$ additional $F$--cost.
But $F(2,2)=22$ and $F(1,2)=30$, hence
\[
m_{2,2}=m_{1,2}=0.
\]
Moreover, $F(6,6)=10$ implies $m_{6,6}=0$, otherwise $P(T)\ge 36>30$.
Thus, in this subcase the only $S$--edges are precisely one $(2,3)$ and two $(3,6)$ edges.

Let $x=m_{1,6}$ and $y=m_{2,6}$.
Counting edges gives
\[
x+y+m_{2,3}+m_{3,6}=n-1,
\]
and the degree--$6$ handshake equation (using $n_4=n_5=0$ and $m_{6,6}=0$) reads
\[
x+y+m_{3,6}=6n_6.
\]
Subtracting yields
\[
-m_{2,3}=6n_6-(n-1).
\]
Since $m_{2,3}=1$, we obtain
\[
n-1\equiv 1\pmod 6
\qquad\text{and therefore}\qquad
n\equiv 2\pmod 6,
\]
contradicting the standing assumption $n\not\equiv 2\pmod 6$.
Hence this configuration is impossible, and we conclude $m_{2,3}=0$.

With $n_4=n_5=0$, $m_{1,3}=0$, and $m_{2,3}=0$, the degree--$3$ handshake reduces to
\[
2m_{3,3}+m_{3,6}=3n_3.
\]
If $m_{3,3}\ge 1$, then $n_3\ge 2$, and the cheapest possibility
$m_{3,3}=1$, $m_{3,6}=4$ yields
\[
P(T)\ge 16+4\cdot 4=32>30,
\]
impossible.
Hence
\[
m_{3,3}=0,
\qquad
m_{3,6}=3n_3\ge 3.
\]
Since $F(3,6)=4$, we have $P(T)\ge 12$.
Under $P(T)\le 30$, it follows that $m_{3,6}\in\{3,6\}$.

Because $m_{3,6}\ge 3$ already contributes at least $12$ to $P(T)$, any positive
$m_{2,2}$ would force $P(T)\ge 12+22= 34>30$, and any positive $m_{1,2}$ would force
$P(T)\ge 12+30= 42>30$.
Thus
\[
m_{2,2}=0,
\qquad
m_{1,2}=0.
\]
Moreover, $P(T)\le 30$ implies $m_{6,6}\le 1$.

Let $x=m_{1,6}$, $y=m_{2,6}$ and $b=m_{6,6}\in\{0,1\}$.
Counting edges gives
\[
x+y+m_{3,6}+b=n-1,
\]
and the degree--$6$ handshake equation reads
\[
x+y+m_{3,6}+2b=6n_6.
\]
Subtracting yields
\[
b=6n_6-(n-1),
\qquad\text{so}\qquad
b\equiv -(n-1)\pmod 6.
\]
Since $b\in\{0,1\}$, this forces
\[
n\equiv 0 \ \text{or}\ 1\pmod 6,
\]
contradicting the assumption $n\not\equiv \{0,1\}\pmod 6$.

Therefore $n_3=0$, completing the proof.
\end{proof}

\begin{lemma}
\label{lem:no3-30-40-final-v5}
Let $\Delta=6$ and let $T$ be a tree on $n$ vertices such that
\[
30<P(T)\le 40
\qquad\text{and}\qquad
n\equiv 3 \ \text{or}\ 4\pmod 6.
\]
Then either $n_3=0$, or $n\equiv 3\pmod 6$ and the following exceptional pattern holds:
\begin{itemize}
\item[\textup{(E1)}] $P(T)=40$, $m_{2,3}=2$, $m_{3,6}=1$, $m_{1,3}=m_{3,3}=0$,
and $m_{i,j}=0$ for all $(i,j)\in S\setminus\{(2,3),(3,6)\}$.
\end{itemize}
\end{lemma}

\begin{proof}
Assume that $n_3\ge 1$.

\medskip
Since $P(T)\le 40$ and $n\not\equiv 0\pmod 6$ (here $n\equiv 3,4\pmod 6$),
Lemma~\ref{lem:no5-threshold} (second item) implies $n_5=0$.
Moreover, as $30<P(T)\le 40$ and $n\not\equiv\{0,1,2\}\pmod 6$ (true for $n\equiv 3,4$),
Lemma~\ref{lem:no4-threshold} (third item) yields $n_4=0$.
Thus
\begin{equation}
n_4=n_5=0.
\label{eq:n4n5-zero-30-40-v5}
\end{equation}

With \eqref{eq:n4n5-zero-30-40-v5}, the degree--$3$ handshake row \eqref{eq:tree-ni-3} becomes
\begin{equation}
m_{1,3}+m_{2,3}+2m_{3,3}+m_{3,6}=3n_3.
\label{eq:H3-30-40-v5}
\end{equation}

\medskip
We first show that
\begin{equation}
m_{1,3}=0.
\label{eq:m13-zero-30-40-v5}
\end{equation}
Indeed, if $m_{1,3}\ge 1$, then $P(T)\ge F(1,3)=24$. If $m_{1,3}\ge 2$, then $P(T)\ge 2F(1,3)>40$, which is a contradiction.
Since $n_3\ge 1$, equation \eqref{eq:H3-30-40-v5} forces at least two additional edges incident to
degree $3$, and under \eqref{eq:n4n5-zero-30-40-v5} each such edge has type $(3,3)$ or $(3,6)$.
If one of them were $(3,3)$, then $P(T)\ge 24+16+4=44>40$, impossible; hence both must be $(3,6)$.
Thus $P(T)\ge 24+2\cdot 4=32$, and the remaining budget to stay $\le 40$ is at most $8$.
No other $(i,j)\in S$ can appear because every additional $S$--edge has $F(i,j)\ge 10$.
Consequently,
\[
m_{1,3}=1,\quad m_{3,6}=2,\quad m_{2,3}=m_{3,3}=0,
\quad\text{and } m_{i,j}=0 \text{ for all }(i,j)\in S\setminus\{(1,3),(3,6)\}.
\]

In particular, $m_{6,6}=0$. Subtracting the edge count equation from the degree–$6$ handshake equation (as in the previous cases) yields $-m_{1,3}=6n_6-(n-1)$, so $n\equiv 2\pmod 6$,
contradicting $n\equiv 3,4\pmod 6$. Hence \eqref{eq:m13-zero-30-40-v5} holds.

\medskip
Now consider $m_{2,3}$.
Since $3F(2,3)=54>40$, we have $m_{2,3}\le 2$.

If $m_{2,3}=2$, then $P(T)\ge 2F(2,3)=36$.
Equation \eqref{eq:H3-30-40-v5} forces at least one additional edge incident to degree $3$,
and the cheapest possibility is $(3,6)$ with cost $4$.
Thus $P(T)\ge 40$, and since $P(T)\le 40$ we must have $P(T)=40$ and exactly one such edge.
Therefore
\[
m_{2,3}=2,\qquad m_{3,6}=1,\qquad m_{1,3}=m_{3,3}=0,
\]
and all other $S$--variables are $0$.
As in the standard congruence check (subtracting the edge count from the degree--$6$ handshake row),
this forces $n\equiv 3\pmod 6$.
Hence, for $n\equiv 3\pmod 6$ we obtain precisely the exceptional pattern \textup{(E1)}.

%It remains to exclude $m_{2,3}\in\{0,1\}$ under our assumptions.

\medskip
We now exclude the case $m_{2,3}=1$.

Assume $m_{2,3}=1$. With $m_{1,3}=0$, equation \eqref{eq:H3-30-40-v5} becomes
\[
1+2m_{3,3}+m_{3,6}=3n_3.
\]

First we claim that $m_{3,3}=0$.
Indeed, if $m_{3,3}\ge 1$, then $n_3\ge 2$ and the cheapest way to satisfy the
above identity is $m_{3,3}=1$ and $m_{3,6}=3$, which yields
\[
P(T)\ \ge\ F(2,3)+F(3,3)+3F(3,6)
\ =\ 18+16+12\ =\ 46\ >40,
\]
a contradiction. Hence $m_{3,3}=0$.

With $m_{3,3}=0$, the degree--$3$ handshake reduces to
\[
1+m_{3,6}=3n_3,
\]
so $m_{3,6}\equiv 2\pmod 3$ and therefore $m_{3,6}\in\{2,5,8,\dots\}$.
Since $P(T)\le 40$ and $F(3,6)=4$, we must have $m_{3,6}\in\{2,5\}$.

\smallskip
\noindent
\emph{Subcase $m_{3,6}=5$.}
Then $n_3=2$ and
\[
P(T)\ \ge\ F(2,3)+5F(3,6)\ =\ 18+20\ =\ 38.
\]
Thus there is room for at most $2$ additional penalty, so no further $(i,j)\in S$
can occur; in particular,
\[
m_{6,6}=m_{2,2}=m_{1,2}=m_{1,3}=m_{3,3}=0.
\]
Let $x=m_{1,6}$ and $y=m_{2,6}$.
Counting edges gives
\[
x+y+m_{2,3}+m_{3,6}=n-1,
\]
while the degree--$6$ handshake row yields (since $m_{6,6}=0$)
\[
x+y+m_{3,6}=6n_6.
\]
Subtracting these two equalities gives
\[
-m_{2,3}=6n_6-(n-1),
\]
so $6n_6-(n-1)=-1$, i.e.\ $n\equiv 2\pmod 6$, contradicting $n\equiv 3,4\pmod 6$.
Hence $m_{3,6}\neq 5$.

\smallskip
\noindent
\emph{Subcase $m_{3,6}=2$.}
Then $n_3=1$ and
\[
P(T)\ \ge\ F(2,3)+2F(3,6)\ =\ 18+8\ =\ 26.
\]
To keep $P(T)\le 40$, we may have $m_{6,6}\le 1$ (since $F(6,6)=10$),
and we must have $m_{2,2}=m_{1,2}=0$ (since $F(2,2)=22$ and $F(1,2)=30$).
Let $b=m_{6,6}\in\{0,1\}$ and again write $x=m_{1,6}$, $y=m_{2,6}$.
Edge counting gives
\[
x+y+m_{2,3}+m_{3,6}+b=n-1,
\]
and the degree--$6$ handshake gives
\[
x+y+m_{3,6}+2b=6n_6.
\]
Subtracting yields
\[
b-m_{2,3}=6n_6-(n-1).
\]
Since $m_{2,3}=1$ and $b\in\{0,1\}$, we get $6n_6-(n-1)\in\{-1,0\}$, hence
$n\equiv 1$ or $2\pmod 6$, again contradicting $n\equiv 3,4\pmod 6$.
Thus $m_{3,6}\neq 2$ as well.

\smallskip
Both subcases lead to contradictions, so $m_{2,3}\neq 1$.
Therefore,
\[
m_{2,3}=0.
\]

\smallskip
\noindent
Finally, assume $m_{2,3}=0$ (and still $m_{1,3}=0$).
Then \eqref{eq:H3-30-40-v5} reduces to
\[
2m_{3,3}+m_{3,6}=3n_3.
\]

\smallskip
\noindent
\textbf{Case 1: $m_{3,3}\ge 2$.}
Then $2m_{3,3}\ge 4$, hence
\[
m_{3,6}=3n_3-2m_{3,3}\le 3n_3-4.
\]
In particular $n_3\ge 3$.
The cheapest realization is $m_{3,3}=2$ and $m_{3,6}=5$, which yields
\[
P(T)\ \ge\ 2F(3,3)+5F(3,6)\ =\ 2\cdot 16 + 5\cdot 4\ =\ 52\ >40,
\]
a contradiction. Hence $m_{3,3}\le 1$.

\smallskip
\noindent
\textbf{Case 2: $m_{3,3}=1$.}
Then $2+m_{3,6}=3n_3$, so $m_{3,6}\equiv 1\pmod 3$ and $n_3\ge 2$.
The smallest possibility is $m_{3,6}=4$, giving
\[
P(T)\ \ge\ F(3,3)+4F(3,6)\ =\ 16+16\ =\ 32.
\]
To stay within $P(T)\le 40$, all remaining $S$–edges must have total cost at most $8$.
Since the smallest positive value in $S$ besides $(3,6)$ is $F(6,6)=10$,
no further $S$–edges are allowed.
Thus $m_{6,6}=m_{2,2}=m_{1,2}=0$.

Let $x=m_{1,6}$ and $y=m_{2,6}$.
Counting edges gives
\[
x+y+5=n-1,
\]
and the degree–$6$ handshake equation yields
\[
x+y+4=6n_6.
\]
Subtracting gives
\[
6n_6-(n-1)=-1,
\qquad\text{so}\qquad
n\equiv 2\pmod 6,
\]
contradicting $n\equiv 3,4\pmod 6$.
Hence this case is impossible.

\smallskip
\noindent
\textbf{Case 3: $m_{3,3}=0$.}
Then $m_{3,6}=3n_3\in\{3,6,9\}$ since $P(T)\le 40$.

\smallskip
\noindent
\textbf{Subcase 3.1: $m_{3,6}=9$.}
Then
\[
P(T)\ \ge\ 9F(3,6)\ =\ 36.
\]
To exceed $30$ while keeping $P(T)\le 40$, we may add at most $4$ more penalty.
But every other $S$–edge has $F(i,j)\ge 10$, so no additional $S$–edges are possible.
Thus $m_{6,6}=m_{2,2}=m_{1,2}=0$.

Let $x=m_{1,6}$ and $y=m_{2,6}$.
Counting edges gives
\[
x+y+9=n-1,
\]
and the degree–$6$ handshake equation yields
\[
x+y+9=6n_6.
\]
Subtracting gives
\[
6n_6-(n-1)=0,
\qquad\text{so}\qquad
n\equiv 1\pmod 6,
\]
contradicting $n\equiv 3,4\pmod 6$.
Hence this subcase is impossible.

\smallskip
\noindent
\textbf{Subcase 3.2: $m_{3,6}=6$.}
Then
\[
P(T)\ \ge\ 6F(3,6)\ =\ 24.
\]
To exceed $30$, we must add at least $7$ more penalty.
The only admissible $S$–edges are $(6,6)$ (cost $10$),
since $(1,2)$ (cost $30$) and  $(2,2)$ (cost $22$) would force $P(T)>40$.

Thus $m_{1,2}=m_{2,2}=0$ and $m_{6,6}\in\{1,2\}$.

Let $x=m_{1,6}$, $y=m_{2,6}$ and $b=m_{6,6}$.
Counting edges gives
\[
x+y+6+b=n-1,
\]
and the degree–$6$ handshake equation yields
\[
x+y+6+2b=6n_6.
\]
Subtracting gives
\[
b=6n_6-(n-1),
\qquad\text{so}\qquad
b\equiv -(n-1)\pmod 6.
\]
If $b=1$, then $n\equiv 0\pmod 6$;  
if $b=2$, then $n\equiv 5\pmod 6$.
Both contradict $n\equiv 3,4\pmod 6$.
Hence this subcase is impossible.

\smallskip
\noindent
\textbf{Subcase 3.3: $m_{3,6}=3$.}
Then $n_3=1$ and the three $(3,6)$–edges contribute $12$ to $P(T)$.
To exceed $30$ while keeping $P(T)\le 40$, we must add between $19$ and $28$ more penalty.
The only admissible $S$–edges are $(6,6)$ (cost $10$) and $(2,2)$ (cost $22)$.

Thus the only possibilities are
\[
(b,c)=(0,1)\quad\text{or}\quad(b,c)=(2,0),
\]
where $b=m_{6,6}$ and $c=m_{2,2}$.

Let $x=m_{1,6}$ and $y=m_{2,6}$.
Counting edges gives
\[
x+y+3+b+c=n-1,
\]
and the degree–$6$ handshake equation yields
\[
x+y+3+2b=6n_6.
\]
Subtracting gives
\[
b-c=6n_6-(n-1),
\qquad\text{so}\qquad
b-c\equiv -(n-1)\pmod 6.
\]
If $(b,c)=(0,1)$, then $b-c=-1$ and $n\equiv 2\pmod 6$.
If $(b,c)=(2,0)$, then $b-c=2$ and $n\equiv 5\pmod 6$.
Both contradict $n\equiv 3,4\pmod 6$.
Hence this subcase is impossible.

\medskip
\noindent
Since all cases with $m_{2,3}=0$ and $n_3\ge 1$ lead to contradictions,
we conclude that $n_3\ge 1$ can occur only in the exceptional pattern \textup{(E1)}
with $n\equiv 3\pmod 6$.
\end{proof}

%========================================================
\section{Extremal configurations for $P(T)$ and $\sigma(T)$}
\label{sec:extremal}

Having eliminated all but a small set of admissible degree patterns, we now proceed
to determine the exact minimum value of $P(T)$ for each residue class of $n$ modulo $6$
and to describe all trees attaining this minimum. Since $\sigma(T)$ differs from a
constant depending only on $n$ by $-P(T)$, this also yields the complete solution to
the problem of maximizing $\sigma(T)$ over $\mathcal{T}_n^{(6)}$.

\bigskip

We determine the minimum value
\[
P_{\min}(n)=\min\{P(T):\,T\in\mathcal{T}_n^{(6)}\}
\]
for every residue class of $n$ modulo $6$ and describe all minimizing trees.
Equivalently, this yields the maximum value of $\sigma(T)$ among all trees in
$\mathcal{T}_n^{(6)}$, since $\sigma(T)$ differs from a constant term depending only on $n$
by $-P(T)$.
The minimum penalty depends only on $n \bmod 6$ and takes exactly six values:
\[
P_{\min}(n)=
\begin{cases}
0,  & n\equiv 1 \pmod 6,\\
10, & n\equiv 0 \pmod 6,\\
20, & n\equiv 5 \pmod 6,\\
22, & n\equiv 2 \pmod 6,\\
30, & n\equiv 4 \pmod 6,\\
40, & n\equiv 3 \pmod 6.
\end{cases}
\]
In the subsequent theorems, for each residue class we give an explicit parameter
description (in terms of $m_{i,j}$ and $n_i$) of all minimizing trees.

\subsection{Parameter descriptions of minimizers}
\label{subsec:param}

% (Place your six theorems here in the current order.)
% Theorems: P0, P10, P20, P22, P30, P40

\begin{theorem}
\label{thm:delta6-P0}
Let $n\ge 7$.
Then
\[
P_{\min}(n)=0
\quad\text{if and only if}\quad
n\equiv 1 \pmod 6.
\]
Moreover, if $T\in\mathcal{T}_n^{( 6)}$ satisfies $P(T)=0$, then
\[
m_{1,6}=\frac{2n+4}{3},\qquad
m_{2,6}=\frac{n-7}{3},\qquad
m_{i,j}=0 \ \text{for all }(i,j)\in S,
\]
and the corresponding degree distribution is given by
\[
n_1=\frac{2n+4}{3},\qquad
n_2=\frac{n-7}{6},\qquad
n_6=\frac{n-1}{6},\qquad
n_3=n_4=n_5=0.
\]
%In particular, the minimizing configuration is unique.
\end{theorem}

\begin{proof}
We establish both implications.

\medskip
We begin by rewriting the relevant variables using the linear constraints imposed by the handshake system.
Let $T\in\mathcal{T}_n^{(6)}$.
From \eqref{eq:m16-general}--\eqref{eq:m26-general},
the quantities $m_{1,6}$ and $m_{2,6}$ can be expressed as linear functions of $n$
and of the remaining variables $m_{i,j}$ with $(i,j)\in S$:
\begin{align}
m_{1,6}
&=
\frac{2n}{3}+\frac{4}{3}
+\sum_{(i,j)\in S} \alpha_{i,j}\,m_{i,j},
\label{eq:m16-general-proof}
\\
m_{2,6}
&=
\frac{n}{3}-\frac{7}{3}
+\sum_{(i,j)\in S} \beta_{i,j}\,m_{i,j}.
\label{eq:m26-general-proof}
\end{align}

Inserting \eqref{eq:m16-general-proof}--\eqref{eq:m26-general-proof} into
$\Sigma(T)=\sum_{1\le i\le j\le 6}\sigma_{i,j}m_{i,j}$, where
$\sigma_{i,j}=(i-j)^2$,
and rearranging terms yields
\[
\Sigma(T)=C(n)-P(T),
\]
with
\[
P(T)=\sum_{(i,j)\in S}F(i,j)\,m_{i,j}.
\]
According to Table~\ref{tab:F-delta6-row}, the coefficients satisfy
\begin{equation}
F(i,j)>0 \quad \text{for all }(i,j)\in S.
\label{eq:F-positive-proof}
\end{equation}

\medskip
This positivity immediately provides an upper bound for $\Sigma(T)$.
Indeed, by \eqref{eq:F-positive-proof} we obtain
\[
\Sigma(T)\le C(n),
\]
and equality holds precisely when all variables indexed by $S$ vanish, that is,
\begin{equation}
m_{i,j}=0 \quad \text{for all }(i,j)\in S.
\label{eq:S-zero-proof}
\end{equation}
Equivalently, this condition is characterized by $P(T)=0$.

\medskip
We next derive the arithmetic restriction on $n$ imposed by the equality case.
Assume that $P(T)=0$ for some $T\in\mathcal{T}_n^{( 6)}$.
Then \eqref{eq:S-zero-proof} holds, and
\eqref{eq:m16-general-proof}--\eqref{eq:m26-general-proof} simplify to
\begin{equation}
m_{1,6}=\frac{2n+4}{3},
\qquad
m_{2,6}=\frac{n-7}{3}.
\label{eq:m16m26-reduced-proof}
\end{equation}
All remaining variables $m_{i,j}$ with $(i,j)\in S$ are therefore zero.
The associated degree counts are given by
\[
n_1=m_{1,6},\qquad
n_2=\frac{m_{2,6}}{2},\qquad
n_6=\frac{m_{1,6}+m_{2,6}}{6},
\qquad
n_3=n_4=n_5=0.
\]
These quantities are integers if and only if $n\equiv 1 \pmod 6$.
Consequently, the condition $P_{\min}(n)=0$ forces $n\equiv 1\pmod 6$.

\medskip
It remains to verify that this congruence condition is also sufficient.
Suppose that $n\equiv 1\pmod 6$.
Then the values in \eqref{eq:m16m26-reduced-proof} are nonnegative integers.
Define
\[
m_{1,6}=\frac{2n+4}{3},\qquad
m_{2,6}=\frac{n-7}{3},\qquad
m_{i,j}=0\ \text{for all }(i,j)\in S.
\]
A direct verification shows that all equations in the handshake system
\eqref{eq:tree-sum-ni}--\eqref{eq:tree-ni-6} are satisfied.
Hence these parameters correspond to a valid tree
$T\in\mathcal{T}_n^{( 6)}$.

For this tree, condition \eqref{eq:S-zero-proof} holds, so $P(T)=0$ and
$\Sigma(T)=C(n)$.

\medskip
Combining the necessary and sufficient parts, we conclude that
$P_{\min}(n)=0$ if and only if $n\equiv 1\pmod 6$.
Moreover, the above argument shows that any tree attaining $P(T)=0$
must realize the stated parameter values, and the minimizing configuration
is therefore unique.
\end{proof}

\begin{theorem}
\label{thm:delta6-P10}
Let $n\ge 12$.
Then
\[
P_{\min}(n)=10
\quad\text{if and only if}\quad
n\equiv 0 \pmod 6.
\]

Moreover, the minimum is attained uniquely by the configuration
\[
m_{6,6}=1,\qquad
m_{1,6}=\frac{2n}{3}+2,\qquad
m_{2,6}=\frac{n-12}{3},
\qquad
m_{i,j}=0\ \text{for all }(i,j)\in S,
\]
and the corresponding degree counts are
\[
n_1=\frac{2n}{3}+2,\qquad
n_2=\frac{n-12}{6},\qquad
n_6=\frac{n}{6},\qquad
n_3=n_4=n_5=0.
\]
\end{theorem}

\begin{proof}
We consider the necessity and sufficiency separately.

\medskip
We first derive restrictions that must hold if the minimum penalty equals $10$.
Assume that
\begin{equation}
P_{\min}(n)=10.
\label{eq:assume-P10}
\end{equation}

By Theorem~\ref{thm:delta6-P0}, the equality $P_{\min}(n)=0$ occurs if and only if
$n\equiv 1\pmod 6$.
Hence, under assumption \eqref{eq:assume-P10}, we necessarily have
\begin{equation}
n\not\equiv 1\pmod 6
\quad\text{and}\quad
P(T)>0 \text{ for all } T\in\mathcal{T}_n^{(6)}.
\label{eq:P-positive}
\end{equation}

Fix a tree $T\in\mathcal{T}_n^{(6)}$ satisfying $P(T)=10$.
We now analyze the structural consequences of this equality.

\medskip
We start by excluding intermediate degrees.
Since $P(T)\le 10$, Lemmas~\ref{lem:no5-threshold},
\ref{lem:no4-threshold}, and \ref{lem:no3-threshold} imply that
\begin{equation}
n_3=n_4=n_5=0.
\label{eq:n345-zero1}
\end{equation}

\medskip
Next, we rule out edge types whose individual contribution exceeds the total
penalty.
From Table~\ref{tab:F-delta6-row} it follows that
$F(1,2)=30>10$ and $F(2,2)=22>10$.
Therefore, the condition $P(T)=10$ forces
\begin{equation}
m_{1,2}=m_{2,2}=0.
\label{eq:m12m22-zero}
\end{equation}

\medskip
Under the constraints \eqref{eq:n345-zero1} and \eqref{eq:m12m22-zero},
the only remaining admissible edge types are $(1,6)$, $(2,6)$, and $(6,6)$.
Since $P(T)>0$ by \eqref{eq:P-positive}, at least one edge indexed by $S$ must be
present.
Among all positive coefficients $F(i,j)$ with $(i,j)\in S$, the smallest value is
$F(6,6)=10$.
Consequently, the entire penalty must be concentrated on a single $(6,6)$–edge,
which yields
\begin{equation}
m_{6,6}=1,
\qquad
m_{i,j}=0 \ \text{for all }(i,j)\in S\setminus\{(6,6)\}.
\label{eq:m66=1}
\end{equation}

\medskip
We now determine the remaining parameters using the handshake equations.
Under the above constraints (\ref{eq:m66=1}), the system reduces to
\[
m_{1,6}+m_{2,6}+1=n-1,
\qquad
m_{1,6}+m_{2,6}+2=6n_6,
\]
together with
\[
n_1=m_{1,6},
\qquad
n_2=\frac{m_{2,6}}{2}.
\]
Solving this system gives
\begin{equation}
m_{1,6}=\frac{2n}{3}+2,
\qquad
m_{2,6}=\frac{n-12}{3}.
\label{eq:m16m26-final}
\end{equation}

These expressions are integers if and only if $n\equiv 0\pmod 6$.
Thus, the condition $P_{\min}(n)=10$ implies
$n\equiv 0\pmod 6$.

\medskip
We now show that this congruence condition is also sufficient.
Assume that $n\equiv 0\pmod 6$.
Define the parameters by \eqref{eq:m16m26-final}, together with
$m_{6,6}=1$ and $m_{i,j}=0$ for all other $(i,j)\in S$.
A direct verification shows that all equations in the handshake system
\eqref{eq:tree-sum-ni}--\eqref{eq:tree-ni-6} are satisfied, and hence these values
correspond to a valid tree $T\in\mathcal{T}_n^{(6)}$.

For this tree we have $P(T)=F(6,6)\,m_{6,6}=10$,
and hence $P_{\min}(n)\le 10$.

On the other hand, suppose that a tree $T'\in\mathcal{T}_n^{(6)}$ satisfies
$P(T')<10$.
Then $P(T') \le 10$, and the same arguments that lead to
\eqref{eq:n345-zero1} and \eqref{eq:m12m22-zero} apply here as well.
Under these constraints, the only admissible edge types are $(1,6)$, $(2,6)$ and
$(6,6)$, and among all coefficients $F(i,j)$ with $(i,j)\in S$ the minimal value is
$F(6,6)=10$.
Therefore any tree with $P(T')>0$ must satisfy $P(T')\ge 10$, while
$P(T')=0$ is impossible since $n\not\equiv 1\pmod 6$ by
\eqref{eq:P-positive}.
Consequently, no tree on $n$ vertices can satisfy $P(T')<10$, and hence $P_{\min}(n)=10$.

\medskip
Combining the necessary and sufficient parts, we conclude that
$P_{\min}(n)=10$ if and only if $n\equiv 0\pmod 6$, and that the minimizing
configuration is unique.
This completes the proof.
\end{proof}

\begin{theorem}
\label{thm:delta6-P20}
Let $n\ge 17$.
Then
\[
P_{\min}(n)=20
\quad\text{if and only if}\quad
n\equiv 5 \pmod 6.
\]
In this case, the minimum is attained uniquely by
\[
m_{6,6}=2,\qquad
m_{1,6}=\frac{2n+8}{3},\qquad
m_{2,6}=\frac{n-17}{3},
\qquad
m_{i,j}=0\ \text{for all }(i,j)\in S\setminus\{(6,6)\},
\]
and the corresponding degree counts satisfy
\[
n_1=\frac{2n+8}{3},\qquad
n_2=\frac{n-17}{6},\qquad
n_6=\frac{n+1}{6},\qquad
n_3=n_4=n_5=0.
\]
\end{theorem}

\begin{proof}
We establish the equivalence by proving necessity and sufficiency.

\medskip
We begin by deriving a necessary arithmetic condition on $n$.
Assume that $P_{\min}(n)=20$, and let $T$ be a tree satisfying $P(T)=20$.

From Theorems~\ref{thm:delta6-P0} and~\ref{thm:delta6-P10}, we know that
$P_{\min}(n)=0$ occurs exactly when $n\equiv 1\pmod 6$, while
$P_{\min}(n)=10$ occurs exactly when $n\equiv 0\pmod 6$.
Hence, under the present assumption, we must have $n\not\equiv 0,1 \pmod 6$.

We next analyze the degree distribution of $T$.
Since $P(T)=20<30$, Lemmas~\ref{lem:no5-threshold} and~\ref{lem:no4-threshold}
exclude the presence of vertices of degree $5$ and $4$.
Moreover, the inequality $10<P(T)\le 22$, together with
$n\not\equiv 0,1\pmod 6$, allows us to invoke
Lemma~\ref{lem:no3-threshold}, which excludes degree~$3$ vertices.
Consequently, $n_3=n_4=n_5=0$.

Under this restriction, the only admissible edge types are
$(1,6)$, $(2,6)$, $(1,2)$, $(2,2)$, and $(6,6)$.
From Table~\ref{tab:F-delta6-row} we have
$F(1,2)=30$ and $F(2,2)=22$, both exceeding the total penalty~$20$.
Therefore neither of these edge types can occur, and we must have $m_{1,2}=m_{2,2}=0$.

As a consequence, the penalty reduces to
$P(T)=10\,m_{6,6}$.

Since $P(T)=20$, this forces $m_{6,6}=2$.

Let $x=m_{1,6}$ and $y=m_{2,6}$.
The total edge count yields
\[
x+y+2=n-1,
\]
while the handshake equation corresponding to degree~$6$ gives
\[
x+y+4=6n_6.
\]
Subtracting these two equations leads to $2=6n_6-(n-1)$,
which is equivalent to the congruence condition
$n\equiv 5 \pmod 6$.

This completes the proof of necessity.

\medskip
We now show that this congruence condition is sufficient.
Assume that $n\equiv 5\pmod 6$, and define
\[
m_{6,6}=2,\qquad
m_{1,6}=\frac{2n+8}{3},\qquad
m_{2,6}=\frac{n-17}{3},
\]
with all other $m_{i,j}$ equal to zero for $(i,j)\in S$.

With $x=m_{1,6}$ and $y=m_{2,6}$, the relation
\[
x+y+2=n-1
\]
verifies the edge count.
Moreover,
\[
n_6=\frac{x+y+4}{6}=\frac{n+1}{6},\qquad
n_2=\frac{y}{2}=\frac{n-17}{6},\qquad
n_1=x=\frac{2n+8}{3},
\]
are all integers, so every equation in the handshake system
\eqref{eq:tree-sum-ni}--\eqref{eq:tree-ni-6} is satisfied.
Thus these parameters define a valid tree $T$.

For this tree we have $P(T)=2\,F(6,6)=20$,
and hence $P_{\min}(n)\le 20$.

On the other hand, suppose that there exists a tree $T'\in\mathcal{T}_n^{(6)}$
with $P(T')<20$.
Then $P(T')\le 20$, and the same arguments used above apply,
implying $n_3=n_4=n_5=m_{1,2}=m_{2,2}=0$.

Under these constraints, the only admissible edge types indexed by $S$ are
$(6,6)$, and the penalty reduces to
\[
P(T')=10\,m_{6,6}.
\]
Since $P(T')>0$ (because $n\not\equiv 0,1\pmod 6$), this implies
$P(T')\ge 10$, with equality only when $m_{6,6}=1$, which would give
$P(T')=10$.
However, by Theorems~\ref{thm:delta6-P0} and~\ref{thm:delta6-P10}, the value
$P_{\min}(n)=10$ occurs if and only if $n\equiv 0\pmod 6$, which contradicts
the present assumption $n\equiv 5\pmod 6$.

Therefore no tree of order $n\equiv 5\pmod 6$ can satisfy $P(T')<20$, and hence
\[
P_{\min}(n)=20.
\]

Finally, uniqueness follows immediately: the condition $P(T)=20$ forces
$m_{6,6}=2$, and all other edge types indexed by $S$ are excluded by the
inequality $F(i,j)>20$.
This completes the proof.
\end{proof}

\begin{theorem}
\label{thm:delta6-P22}
Let $n\ge 14$. Then
\[
P_{\min}(n)=22
\quad\text{if and only if}\quad
n\equiv 2 \pmod 6.
\]
In this case, the minimum is attained uniquely by
\[
m_{2,2}=1,\qquad
m_{1,6}=\frac{2n+2}{3},\qquad
m_{2,6}=\frac{n-8}{3},
\qquad
m_{i,j}=0\ \text{for all }(i,j)\in S\setminus\{(2,2)\},
\]
and the corresponding degree counts satisfy
\[
n_1=\frac{2n+2}{3},\qquad
n_2=\frac{n-2}{6},\qquad
n_6=\frac{n-2}{6},\qquad
n_3=n_4=n_5=0.
\]
\end{theorem}

\begin{proof}
We again prove necessity and sufficiency.

\medskip
Assume that $P_{\min}(n)=22$, and let $T$ be a tree with $P(T)=22$.
From Theorems~\ref{thm:delta6-P0}, \ref{thm:delta6-P10}, and
\ref{thm:delta6-P20}, we know that the values $P_{\min}(n)=0,10,20$ occur
precisely when $n\equiv 1,0,5\pmod 6$, respectively.
Hence, under the present assumption,
\[
n\not\equiv 0,1,5 \pmod 6.
\]

We next analyze the degree distribution of $T$.
Since $P(T)=22<30$, Lemmas~\ref{lem:no5-threshold} and~\ref{lem:no4-threshold}
exclude vertices of degrees $5$ and $4$.
Moreover, the inequality $10<P(T)\le 22$, together with
$n\not\equiv 0,1\pmod 6$, allows us to apply
Lemma~\ref{lem:no3-threshold}, which excludes degree~$3$ vertices.
Consequently, $n_3=n_4=n_5=0$.

Under this restriction, only degrees $1$, $2$, and $6$ may occur.
The admissible edge types are therefore $(1,6)$, $(2,6)$, $(1,2)$, $(2,2)$,
and $(6,6)$.
From Table~\ref{tab:F-delta6-row} we have $F(1,2)=30>22$, which forces $m_{1,2}=0$.

Let
\[
x=m_{1,6},\qquad
y=m_{2,6},\qquad
a=m_{2,2},\qquad
b=m_{6,6}.
\]
Since $F(6,6)=10$ and $F(2,2)=22$, the penalty can be written as
\begin{equation}
P(T)=10\,b+22\,a.
\label{eq:P22-decomp}
\end{equation}
According to (\ref{eq:P22-decomp}), the equality $P(T)=22$ has the unique nonnegative integer solution
\begin{equation}
a=1,\qquad b=0.
\label{eq:a1b0}
\end{equation}

We now determine the arithmetic restriction on $n$ imposed by the handshake
constraints.
The edge count gives
\[
x+y+a+b=n-1,
\]
while the handshake equation for degree~$6$ yields
\[
x+y+2b=6n_6.
\]
Subtracting these two equations and using \eqref{eq:a1b0} leads to $-1=6n_6-(n-1)$,
which is equivalent to
$n\equiv 2 \pmod 6$.
This establishes the necessary condition.

\medskip
We now show that this congruence condition is sufficient.
Assume that $n\equiv 2\pmod 6$, and define
\[
m_{2,2}=1,\qquad
m_{1,6}=\frac{2n+2}{3},\qquad
m_{2,6}=\frac{n-8}{3},
\]
with all other $m_{i,j}$ equal to zero for $(i,j)\in S$.

With $x=m_{1,6}$ and $y=m_{2,6}$, the identity
\[
x+y+1=n-1
\]
verifies the edge count.
Moreover,
\[
n_6=\frac{x+y}{6}=\frac{n-2}{6},\qquad
n_2=\frac{2+2y}{2}=\frac{n-2}{6},\qquad
n_1=x=\frac{2n+2}{3},
\]
are all integers, and hence every equation in the handshake system
\eqref{eq:tree-sum-ni}--\eqref{eq:tree-ni-6} is satisfied.
Thus these parameters define a valid tree $T$.

For this tree we have
\[
P(T)=F(2,2)=22,
\]
and hence $P_{\min}(n)\le 22$.

On the other hand, suppose that there exists a tree $T'\in\mathcal{T}_n^{(6)}$
with $P(T')<22$.
Then $P(T')\le 22$, and the same structural restrictions derived above apply.
In particular, Lemmas~\ref{lem:no5-threshold}, \ref{lem:no4-threshold},
and~\ref{lem:no3-threshold} imply $n_3=n_4=n_5=0$,
and from $F(1,2)=30>22$ it follows that
$m_{1,2}=0$.

Under these constraints, the penalty has the form
\[
P(T')=10\,m_{6,6}+22\,m_{2,2}.
\]
If $m_{2,2}\ge 1$, then $P(T')\ge 22$, so $m_{2,2}=0$ must hold.
Thus
\[
P(T')=10\,m_{6,6},
\]
and since $P(T')>0$, we obtain $P(T')\in\{10,20\}$.

If $P(T')=10$, then by Theorem~\ref{thm:delta6-P10} we must have
$n\equiv 0\pmod 6$, contradicting $n\equiv 2\pmod 6$.
If $P(T')=20$, then by Theorem~\ref{thm:delta6-P20} we must have
$n\equiv 5\pmod 6$, again contradicting $n\equiv 2\pmod 6$.
Therefore no tree of order $n\equiv 2\pmod 6$ can satisfy $P(T')<22$, and hence
\[
P_{\min}(n)=22.
\]

Finally, uniqueness follows immediately from
\eqref{eq:P22-decomp}–\eqref{eq:a1b0}, which force $(m_{2,2},m_{6,6})=(1,0)$.
This completes the proof.
\end{proof}

\begin{theorem}
\label{thm:delta6-P30}
Let $n\ge 22$.
Then
\[
P_{\min}(n)=30
\quad\text{if and only if}\quad
n\equiv 4 \pmod 6.
\]
In this case, the minimum is attained uniquely by
\[
m_{6,6}=3,\qquad
m_{1,6}=\frac{2n+10}{3},\qquad
m_{2,6}=\frac{n-22}{3},
\qquad
m_{i,j}=0\ \text{for all }(i,j)\in S\setminus\{(6,6)\},
\]
and the corresponding degree counts satisfy
\[
n_1=\frac{2n+10}{3},\qquad
n_2=\frac{n-22}{6},\qquad
n_6=\frac{n+2}{6},\qquad
n_3=n_4=n_5=0.
\]
\end{theorem}

\begin{proof}
We again derive the necessary and sufficient conditions.

\medskip
Assume that $P_{\min}(n)=30$, and let $T$ be a tree satisfying $P(T)=30$.
From Theorems~\ref{thm:delta6-P0}, \ref{thm:delta6-P10},
\ref{thm:delta6-P20}, and \ref{thm:delta6-P22}, the smaller penalty values
$0$, $10$, $20$, and $22$ are already completely characterized by the
congruence classes
$n\equiv 1,0,5,$ and $2\pmod 6$, respectively.
Hence, under the present assumption,
\[
n\not\equiv 0,1,2,5 \pmod 6.
\]

We now examine the degree distribution of $T$.
Since $P(T)=30$, Lemma~\ref{lem:no5-threshold} excludes vertices of degree $5$.
Moreover, because $n\not\equiv 1\pmod 6$, Lemma~\ref{lem:no4-threshold}
(second case) excludes vertices of degree $4$.
Finally, as $P(T)>22$ and $n\not\equiv 0,1,2\pmod 6$, Lemma~\ref{lem:no3-22-30}
implies that vertices of degree $3$ cannot occur.
Consequently, $n_3=n_4=n_5=0$.

Under this restriction, only degrees $1$, $2$, and $6$ are present.
Thus the admissible edge types are
$(1,6)$, $(2,6)$, $(1,2)$, $(2,2)$, and $(6,6)$.
From Table~\ref{tab:F-delta6-row} we have
$F(1,2)=30$, $F(2,2)=22$, and $F(6,6)=10$.
Therefore the penalty can be written as
\begin{equation}
P(T)=30\,m_{1,2}+22\,m_{2,2}+10\,m_{6,6}.
\label{eq:P30-decomp}
\end{equation}

Solving \eqref{eq:P30-decomp} in nonnegative integers under the constraint
$P(T)=30$ yields exactly two candidates:
\[
(m_{1,2},m_{2,2},m_{6,6})=(1,0,0)
\quad\text{or}\quad
(0,0,3).
\]

We show that the first option is impossible.
Assume $m_{1,2}=1$, and let $x=m_{1,6}$ and $y=m_{2,6}$.
The edge count gives
\[
x+y+1=n-1,
\]
while the degree--$6$ handshake equation yields
\[
x+y=6n_6.
\]
Subtracting these relations leads to $-1=6n_6-(n-1)$,
which implies $n\equiv 2\pmod 6$, contradicting
$n\not\equiv 0,1,2,5\pmod 6$.
Hence this case cannot occur.

Therefore,
\begin{equation}
m_{6,6}=3,
\qquad
m_{1,2}=m_{2,2}=0.
\label{eq:m66=3}
\end{equation}

As before, let $x = m_{1,6}$ and $y = m_{2,6}$.
The edge count now gives
\[
x+y+3=n-1,
\]
and the degree--$6$ handshake equation becomes 
\[
x+y+6=6n_6.
\]
Subtracting these equations yields $3=6n_6-(n-1)$,
which is equivalent to
$n\equiv 4 \pmod 6$.
This establishes the necessary condition.

\medskip
We now show that this congruence condition is sufficient.
Assume that $n\equiv 4\pmod 6$, and define
\[
m_{6,6}=3,\qquad
m_{1,6}=\frac{2n+10}{3},\qquad
m_{2,6}=\frac{n-22}{3},
\]
with all other $m_{i,j}$ equal to zero for $(i,j)\in S$.

With these choices, the edge count and the degree--$6$ handshake equation are
satisfied, and
\[
n_1=\frac{2n+10}{3},\qquad
n_2=\frac{n-22}{6},\qquad
n_6=\frac{n+2}{6}
\]
are nonnegative integers.
Thus all equations in the handshake system
\eqref{eq:tree-sum-ni}--\eqref{eq:tree-ni-6} hold, and the parameters define a
valid tree $T$.

For this tree,
\[
P(T)=3\,F(6,6)=30,
\]
so $P_{\min}(n)\le 30$.
On the other hand, by \eqref{eq:P30-decomp} the only smaller values that
$P(T)$ can attain are $0$, $10$, $20$, and $22$, which correspond precisely to
the excluded congruence classes.
Hence no tree of order $n\equiv 4\pmod 6$ can satisfy $P(T)<30$, and therefore
$P_{\min}(n)=30$.

Finally, uniqueness follows from \eqref{eq:P30-decomp} together with
\eqref{eq:m66=3}, which force $m_{6,6}=3$ and exclude all other possibilities.
This completes the proof.
\end{proof}

\begin{theorem}
\label{thm:delta6-P40}
Let $n$ be a positive integer. Then
\[
P_{\min}(n)=40
\quad\text{if and only if}\quad
n\equiv 3 \pmod 6.
\]
Moreover, if $n\equiv 3\pmod 6$, then the minimum $P_{\min}(n)=40$ is attained
exactly by the following two families of solutions (written in terms of the
nonzero $m_{i,j}$, with all other $(i,j)\in S$ equal to $0$):
\begin{itemize}
\item[\textup{(i)}] (\emph{the case $n_3=0$}, valid for $n\ge 27$)
\[
m_{6,6}=4,\qquad
m_{1,6}=\frac{2n+12}{3},\qquad
m_{2,6}=\frac{n-27}{3};
\]
\item[\textup{(ii)}] (valid for $n\ge 15$)
\[
m_{2,3}=2,\qquad
m_{3,6}=1,\qquad
m_{1,6}=\frac{2n+3}{3},\qquad
m_{2,6}=\frac{n-15}{3}.
\]
\end{itemize}
\end{theorem}

\begin{proof}
Assume that $P_{\min}(n)=40$.
By Theorems~\ref{thm:delta6-P0}, \ref{thm:delta6-P10},
\ref{thm:delta6-P20}, \ref{thm:delta6-P22}, and \ref{thm:delta6-P30},
the smaller minimum penalties $0$, $10$, $20$, $22$, and $30$ occur precisely for
$n\equiv 1,0,5,2,$ and $4\pmod 6$, respectively.
Hence $n\not\equiv 0,1,2,4,5\pmod 6$, and therefore $n\equiv 3\pmod 6$.

\medskip
Assume from now on that $n\equiv 3\pmod 6$.
Since the value $30$ occurs only when $n\equiv 4\pmod 6$, we have
\begin{equation}
P_{\min}(n)>30.
\label{eq:Pgt30}
\end{equation}
Thus it suffices to show that there exists a tree $T$ with $P(T)=40$ and that
no tree satisfies $30<P(T)<40$.

\medskip
Let $T$ be any tree with $30<P(T)\le 40$ and $n\equiv 3\pmod 6$.
Since $n\not\equiv 0\pmod 6$ and $P(T)\le 40$,
Lemma~\ref{lem:no5-threshold} (second item) implies $n_5=0$.
Moreover, since $n\not\equiv 0,1,2\pmod 6$ and $30<P(T)\le 40$,
Lemma~\ref{lem:no4-threshold} (third item) yields $n_4=0$.
Hence
\begin{equation}
n_4=n_5=0.
\label{eq:n4n5zero-40}
\end{equation}

Under \eqref{eq:n4n5zero-40}, we apply Lemma~\ref{lem:no3-30-40-final-v5}.
Since $n\equiv 3\pmod 6$ and $30<P(T)\le 40$, either $n_3=0$, or $T$ satisfies the
exceptional pattern \textup{(E1)} from Lemma~\ref{lem:no3-30-40-final-v5}.
In particular, if $n_3>0$, then necessarily $P(T)=40$ and $T$ is of type
\textup{(ii)} (the pattern \textup{(E1)}).
Therefore there is no tree with $n_3>0$ and $30<P(T)<40$, and it remains only to
exclude $30<P(T)<40$ in the case
\begin{equation}
n_3=0.
\label{eq:n3zero-40}
\end{equation}

Assume \eqref{eq:n3zero-40}.
Together with \eqref{eq:n4n5zero-40}, only degrees $1$, $2$, and $6$ may occur,
so the admissible edge types are $(1,6)$, $(2,6)$, $(1,2)$, $(2,2)$, and $(6,6)$.
Let
\[
x=m_{1,6},\qquad y=m_{2,6},\qquad a=m_{1,2},\qquad c=m_{2,2},\qquad b=m_{6,6}.
\]
From Table~\ref{tab:F-delta6-row} we have $F(1,2)=30$, $F(2,2)=22$, and $F(6,6)=10$,
and hence
\begin{equation}
P(T)=30a+22c+10b.
\label{eq:P-126}
\end{equation}
Counting edges gives
\begin{equation}
x+y+a+c+b=n-1,
\label{eq:Ecount-40}
\end{equation}
while the degree--$6$ handshake equation yields
\begin{equation}
x+y+2b=6n_6.
\label{eq:H6-40}
\end{equation}
Subtracting \eqref{eq:Ecount-40} from \eqref{eq:H6-40} gives
\begin{equation}
b-a-c=6n_6-(n-1).
\label{eq:mod-core-40}
\end{equation}
Reducing \eqref{eq:mod-core-40} modulo $6$ and using $n\equiv 3\pmod 6$
(hence $n-1\equiv 2\pmod 6$) yields
\begin{equation}
b-a-c\equiv 4 \pmod 6.
\label{eq:bac-cong-40}
\end{equation}

If $P(T)<40$, then by \eqref{eq:P-126} the only possible values are
\[
P(T)\in\{0,10,20,22,30,32\}.
\]
When $P(T)\le 30$, the values $0$, $10$, $20$, $22$, and $30$ occur only for
$n\equiv 1,0,5,2,4\pmod 6$, respectively, impossible since $n\equiv 3\pmod 6$.
For $P(T)=32$, equation \eqref{eq:P-126} forces $(a,c,b)=(0,1,1)$, and then
$b-a-c=0$ contradicts \eqref{eq:bac-cong-40}.
Therefore $P(T)\ge 40$ whenever \eqref{eq:n3zero-40} holds.

Moreover, combining \eqref{eq:bac-cong-40} with $P(T)=40$ forces $(a,c,b)=(0,0,4)$,
and hence, in the case $n_3=0$ we must have
\begin{equation}
m_{6,6}=4,\qquad m_{1,2}=m_{2,2}=0.
\label{eq:C0-core}
\end{equation}

\medskip
We now show that $P(T)=40$ is attainable.
For $n\equiv 3\pmod 6$, the exceptional pattern \textup{(ii)} satisfies the handshake
equations and yields $P(T)=40$, hence $P_{\min}(n)\le 40$.
In addition, in case \textup{(i)} (where $n_3=0$), the constraints
\eqref{eq:C0-core} and \eqref{eq:Ecount-40} uniquely yield
\[
m_{1,6}=\frac{2n+12}{3},\qquad m_{2,6}=\frac{n-27}{3},
\]
so \textup{(i)} also attains $P(T)=40$ whenever these values are nonnegative integers.

\medskip
Finally, \eqref{eq:Pgt30} excludes $P_{\min}(n)\le 30$, and we have shown that
no tree satisfies $30<P(T)<40$ while the families \textup{(i)} and \textup{(ii)}
attain $P(T)=40$. Hence $P_{\min}(n)=40$ for all $n\equiv 3\pmod 6$, and the
minimum is attained exactly by \textup{(i)} and \textup{(ii)}.
\end{proof}

\subsection{Explicit constructions of minimizers}
\label{subsec:construct}

\medskip

Each of the following theorems provides a complete parameter description of the
minimizing configurations in terms of the variables $m_{i,j}$ and $n_i$. To
complement these algebraic characterizations, we next give explicit constructive
descriptions of all minimizing trees. Each construction lemma defines a natural family of trees and shows
that a tree $T\in\mathcal{T}_n^{(6)}$ attains the stated minimum penalty if and only if
$T$ belongs to that family (up to isomorphism).

% (Place your construction lemmas after the corresponding theorem, as in your current skeleton.)
%========================================================

\medskip

Let $k\ge 1$.
Let $TT^1_{\mathrm{opt}}(k)$ denote the tree obtained from the path
$P_{2k+1}=v_1v_2\cdots v_{2k+1}$ by attaching exactly four pendant vertices to
each even vertex $v_{2},v_{4},\dots,v_{2k}$.

\begin{lemma}
\label{lem:TT1opt-delta6}
Let $k\ge 1$ and set $n=6k+1$.
Then $TT^1_{\mathrm{opt}}(k)$ has maximum degree $6$ and satisfies $P(TT^1_{\mathrm{opt}}(k))=0$.
Moreover, a tree $T\in\mathcal{T}_{n}^{(6)}$ satisfies $P(T)=0$ if and only if
$T\cong TT^1_{\mathrm{opt}}(k)$.
\end{lemma}

\begin{proof}
Fix $k\ge 1$ and let $n=6k+1$.
By Theorem~\ref{thm:delta6-P0}, the equality $P(T)=0$ is possible if and only if
$n\equiv 1\pmod 6$, and in this case the minimizing configuration is unique.
Moreover, if $T\in\mathcal{T}_n^{(6)}$ satisfies $P(T)=0$, then
\begin{equation}
m_{1,6}=4k+2,\qquad
m_{2,6}=2k-2,\qquad
m_{i,j}=0\ \text{for all }(i,j)\in S,
\label{eq:TT-counts-m}
\end{equation}
and the corresponding degree counts satisfy
\begin{equation}
n_1=4k+2,\qquad
n_2=k-1,\qquad
n_6=k,\qquad
n_3=n_4=n_5=0.
\label{eq:TT-counts-n}
\end{equation}

\medskip
We first note that $TT^1_{\mathrm{opt}}(k)$ satisfies
\eqref{eq:TT-counts-m}--\eqref{eq:TT-counts-n}.
Indeed, every even vertex on the defining path has degree $6$ (two path neighbors
and four pendant neighbors), every internal odd vertex has degree $2$, and the
two endpoints have degree $1$.
Hence $n_6=k$, $n_2=k-1$, $n_1=4k+2$, and all edges are of type $(1,6)$ or $(2,6)$,
with exactly $m_{1,6}=4k+2$ and $m_{2,6}=2k-2$.
By Theorem~\ref{thm:delta6-P0}, this implies $P(TT^1_{\mathrm{opt}}(k))=0$.

\medskip
It remains to show uniqueness.
Let $T\in\mathcal{T}_n^{(6)}$ satisfy $P(T)=0$, so that
\eqref{eq:TT-counts-m}--\eqref{eq:TT-counts-n} hold.
We prove by induction on $k$ that $T\cong TT^1_{\mathrm{opt}}(k)$.

For $k=1$, we have $n=7$ and $(n_1,n_2,n_6)=(6,0,1)$, hence $T$ is the star $S_7$,
which coincides with $TT^1_{\mathrm{opt}}(1)$.

Assume the statement holds for $k$ and consider a tree $T$ of order $6(k+1)+1$
satisfying $P(T)=0$.
Then every edge is of type $(1,6)$ or $(2,6)$, and every degree--$6$ vertex is
adjacent only to vertices of degree $1$ or $2$.
Let $u$ be an endpoint of a longest path in $T$ and let $v$ be its neighbor.
Then $u$ is a leaf and $\deg(v)=6$.
A standard longest–path argument shows that among the remaining neighbors of $v$,
exactly four are leaves and exactly one has degree $2$.

Removing $v$ together with its four pendant neighbors produces a tree $T'$ in
which the unique remaining neighbor of $v$ becomes a leaf.
The resulting tree has parameters
\[
n'_6=k,\qquad n'_2=k-1,\qquad n'_1=4k+2,
\]
and all other degrees equal to zero, hence again satisfies
\eqref{eq:TT-counts-m}--\eqref{eq:TT-counts-n} with parameter $k$.
By the induction hypothesis, $T'\cong TT^1_{\mathrm{opt}}(k)$.

Reattaching the removed configuration corresponds exactly to extending the
defining path of $TT^1_{\mathrm{opt}}(k)$ by one additional even vertex with four
pendant neighbors.
Thus $T\cong TT^1_{\mathrm{opt}}(k+1)$.

\medskip
Therefore, every tree $T\in\mathcal{T}_{6k+1}^{(6)}$ with $P(T)=0$ is isomorphic to
$TT^1_{\mathrm{opt}}(k)$.
Since $TT^1_{\mathrm{opt}}(k)$ itself satisfies $P(T)=0$, the proof is complete.
\end{proof}

\begin{lemma}
\label{lem:TT6opt-delta6}
Let $k\ge 1$ and set $n=6k+6$.
Let $TT^6_{\mathrm{opt}}(k)$ denote the family of trees obtained from
$TT^1_{\mathrm{opt}}(k)$ by subdividing exactly one edge $v_iv_{i+1}$ with odd
$i$ and $3\le i\le 2k-1$, and then attaching four pendant vertices to the new
subdivision vertex.

Then a tree $T\in\mathcal{T}^{(6)}_n$ satisfies $P(T)=10$
if and only if $T\in TT^6_{\mathrm{opt}}(k)$.
In particular, the unique minimizing configuration from
Theorem~\ref{thm:delta6-P10} belongs to $TT^6_{\mathrm{opt}}(k)$.
\end{lemma}

\begin{proof}
Fix $k\ge 1$ and let $n=6k+6$.
By Theorem~\ref{thm:delta6-P10}, a tree $T\in\mathcal{T}^{(6)}_n$ satisfies
$P(T)=10$ if and only if $n\equiv 0\pmod 6$, and in this case the minimizing
configuration is unique.
Moreover, such a tree satisfies
\begin{equation}
m_{6,6}=1,\qquad
m_{1,6}=4k+6,\qquad
m_{2,6}=2k-2,
\label{eq:TT6-mcounts}
\end{equation}
and
\begin{equation}
n_1=4k+6,\qquad
n_2=k-1,\qquad
n_6=k+1,\qquad
n_3=n_4=n_5=0.
\label{eq:TT6-ncounts}
\end{equation}

\medskip

We first verify that every tree in $TT^6_{\mathrm{opt}}(k)$ satisfies
\eqref{eq:TT6-mcounts}--\eqref{eq:TT6-ncounts}.
Start from $TT^1_{\mathrm{opt}}(k)$ and subdivide an edge $v_iv_{i+1}$ with odd
$i$ and $3\le i\le 2k-1$.
In $TT^1_{\mathrm{opt}}(k)$ such an edge is of type $(2,6)$, so after subdivision
we obtain a new vertex $s$ adjacent to one degree--$2$ and one degree--$6$ vertex.
Attaching four pendant vertices to $s$ increases $\deg(s)$ from $2$ to $6$ and adds
four leaves.
Thus the number of degree--$6$ vertices increases by one, the number of leaves
increases by four, and the number of degree--$2$ vertices is unchanged.
Moreover, exactly one edge of type $(6,6)$ is created (namely the edge joining $s$
to its degree--$6$ neighbor), while all other edges remain of type $(1,6)$ or $(2,6)$.
Consequently \eqref{eq:TT6-mcounts}--\eqref{eq:TT6-ncounts} hold, and by
Theorem~\ref{thm:delta6-P10} such a tree satisfies $P(T)=10$.

\medskip

Conversely, let $T\in\mathcal{T}^{(6)}_n$ satisfy $P(T)=10$.
Then \eqref{eq:TT6-mcounts}--\eqref{eq:TT6-ncounts} hold, and in particular $T$
contains a unique edge of type $(6,6)$.
Let $s$ be the endpoint of this edge that is adjacent to four pendant vertices
(the existence and uniqueness of such a vertex follows from \eqref{eq:TT6-ncounts}).

Remove the four pendant vertices adjacent to $s$ and then suppress $s$, that is,
delete $s$ and replace its two remaining incident edges by a single edge.
The resulting tree $T^\circ$ has order $6k+1$, maximum degree $6$, and contains
only edges of type $(1,6)$ and $(2,6)$.
Thus $T^\circ$ satisfies the parameter description of
Theorem~\ref{thm:delta6-P0}, and by Lemma~\ref{lem:TT1opt-delta6} we obtain
\[
T^\circ \cong TT^1_{\mathrm{opt}}(k).
\]

Reversing the suppression shows that $T$ is obtained from
$TT^1_{\mathrm{opt}}(k)$ by subdividing exactly one $(2,6)$-edge on the defining
path, i.e., an edge $v_iv_{i+1}$ with odd $i$ and $3\le i\le 2k-1$, and attaching
four pendant vertices to the new vertex.
Therefore $T\in TT^6_{\mathrm{opt}}(k)$.
\end{proof}

\begin{lemma}
\label{lem:TT5opt-delta6}
Let $k\ge 2$ and set $n=6k+5$.
Let $TT^5_{\mathrm{opt}}(k)$ denote the family of trees obtained from a tree
$T_6\in TT^6_{\mathrm{opt}}(k-1)$ by subdividing an arbitrary edge of $T_6$
by a new vertex $s$ and then attaching four pendant vertices to $s$.

Then a tree $T\in\mathcal{T}^{(6)}_n$ satisfies $P(T)=20$
if and only if $T\in TT^5_{\mathrm{opt}}(k)$.
In particular, the unique minimizing configuration from
Theorem~\ref{thm:delta6-P20} belongs to $TT^5_{\mathrm{opt}}(k)$.
\end{lemma}

\begin{proof}
Fix $k\ge 2$ and let $n=6k+5$.
By Theorem~\ref{thm:delta6-P20}, a tree $T\in\mathcal{T}^{(6)}_n$ satisfies
$P(T)=20$ if and only if $n\equiv 5\pmod 6$, and in this case the minimizing
configuration is unique.
Moreover, such a tree is characterized by
\begin{equation}
m_{6,6}=2,\qquad
m_{1,6}=4k+6,\qquad
m_{2,6}=2k-4,
\label{eq:TT5-mcounts}
\end{equation}
and by the degree frequencies
\begin{equation}
n_1=4k+6,\qquad
n_2=k-2,\qquad
n_6=k+1,\qquad
n_3=n_4=n_5=0.
\label{eq:TT5-ncounts}
\end{equation}

Let $T_6\in TT^6_{\mathrm{opt}}(k-1)$.
Then $|V(T_6)|=6k$ and $P(T_6)=10$.
Choose any edge $e$ of $T_6$, subdivide it by a new vertex $s$, and attach four
pendant vertices to $s$.
The subdivision replaces the edge $e$ by two edges incident with $s$, and the
four pendant attachments increase $\deg(s)$ to $6$.
Hence the order increases by $5$, the maximum degree remains $6$, and a new
degree--$6$ vertex is created together with four new leaves.

If the subdivided edge was of type $(6,6)$, then this edge is destroyed and two
new $(6,6)$-edges incident with $s$ are created; otherwise exactly one new
$(6,6)$-edge is created.
As a result, $m_{6,6}$ increases by exactly one,
$m_{1,6}$ increases by four, and $m_{2,6}$ remains unchanged.
Since $T_6$ satisfies $m_{6,6}=1$, $m_{1,6}=4k+2$, and $m_{2,6}=2k-4$, the resulting
tree $T$ satisfies \eqref{eq:TT5-mcounts}--\eqref{eq:TT5-ncounts}.
By Theorem~\ref{thm:delta6-P20}, this implies $P(T)=20$, and therefore
$T\in TT^5_{\mathrm{opt}}(k)$.

Conversely, let $T\in\mathcal{T}^{(6)}_n$ satisfy $P(T)=20$.
Then \eqref{eq:TT5-mcounts}--\eqref{eq:TT5-ncounts} hold.
In particular, $T$ contains exactly two edges of type $(6,6)$.
Hence the subgraph induced by the degree--$6$ vertices has a vertex $s$ incident
with exactly one $(6,6)$-edge, and the degree constraints force $s$ to be adjacent
to exactly four pendant vertices.

Remove the four pendant neighbors of $s$ and suppress $s$, that is, delete $s$
and replace its two remaining incident edges by a single edge.
Denote the resulting tree by $T^\circ$.
This operation reduces the order by $5$, so $|V(T^\circ)|=6k$, and removes exactly
one $(6,6)$-edge.
Consequently, $T^\circ$ satisfies the parameter description of
Theorem~\ref{thm:delta6-P10}.
By Lemma~\ref{lem:TT6opt-delta6}, we obtain
\[
T^\circ \in TT^6_{\mathrm{opt}}(k-1).
\]

Reversing the suppression shows that $T$ is obtained from $T^\circ$ by subdividing
one edge and attaching four pendant vertices to the subdivision vertex.
Therefore $T\in TT^5_{\mathrm{opt}}(k)$.
\end{proof}

\begin{lemma}
\label{lem:TT2opt-delta6}
Let $k\ge 2$ and set $n=6k+2$.
Let $TT^2_{\mathrm{opt}}(k)$ denote the family of trees obtained from
$TT^1_{\mathrm{opt}}(k)$ by subdividing exactly one edge $v_iv_{i+1}$ with odd
$i$ and $3\le i\le 2k-1$.

Then a tree $T\in\mathcal{T}^{(6)}_n$ satisfies $P(T)=22$
if and only if $T\in TT^2_{\mathrm{opt}}(k)$.
In particular, every tree satisfying the parameter description in
Theorem~\ref{thm:delta6-P22} is isomorphic to a member of $TT^2_{\mathrm{opt}}(k)$.
\end{lemma}

\begin{proof}
Fix $k\ge 2$ and let $n=6k+2$.
By Theorem~\ref{thm:delta6-P22}, the equality $P(T)=22$ is possible only for
$n\equiv 2\pmod 6$, and in this case the minimizing configuration is unique.
In particular, if $T\in\mathcal{T}^{(6)}_n$ satisfies $P(T)=22$, then
\begin{equation}
m_{2,2}=1,\qquad
m_{1,6}=4k+2,\qquad
m_{2,6}=2k-2,
\label{eq:TT2-mcounts}
\end{equation}
and the degree counts satisfy
\begin{equation}
n_1=4k+2,\qquad
n_2=k,\qquad
n_6=k,\qquad
n_3=n_4=n_5=0.
\label{eq:TT2-ncounts}
\end{equation}

\medskip
We first observe that every tree in $TT^2_{\mathrm{opt}}(k)$ satisfies
\eqref{eq:TT2-mcounts}--\eqref{eq:TT2-ncounts}.
Indeed, starting from $TT^1_{\mathrm{opt}}(k)$ and subdividing one edge
$v_iv_{i+1}$ with odd $i$ creates exactly one new vertex of degree $2$ and
introduces a single edge of type $(2,2)$, while all other degrees and edge types
remain unchanged.
Consequently, the resulting tree has exactly one $(2,2)$–edge, all remaining
edges of type $(1,6)$ or $(2,6)$, and the parameters given above.
By Theorem~\ref{thm:delta6-P22}, this implies $P(T)=22$.

\medskip
Conversely, let $T\in\mathcal{T}^{(6)}_n$ satisfy $P(T)=22$.
Then \eqref{eq:TT2-mcounts}--\eqref{eq:TT2-ncounts} hold, and in particular $T$
contains a unique edge of type $(2,2)$.
Contract this edge to a single degree–$2$ vertex, obtaining a tree $T^\circ$.
This operation reduces the order by one, eliminates the unique $(2,2)$–edge,
and leaves only edges of type $(1,6)$ and $(2,6)$.
Hence $T^\circ$ has order $6k+1$ and degree counts
\[
n^\circ_1=4k+2,\qquad n^\circ_2=k-1,\qquad n^\circ_6=k,
\qquad n^\circ_3=n^\circ_4=n^\circ_5=0.
\]
By Theorem~\ref{thm:delta6-P0} and Lemma~\ref{lem:TT1opt-delta6}, we obtain
$T^\circ\cong TT^1_{\mathrm{opt}}(k)$.

Reversing the contraction, the tree $T$ is obtained from
$TT^1_{\mathrm{opt}}(k)$ by subdividing exactly one edge of the defining path
$P_{2k+1}$, namely an edge $v_iv_{i+1}$ with odd $i$ and $3\le i\le 2k-1$.
Hence $T\in TT^2_{\mathrm{opt}}(k)$.

\medskip
This completes the proof.
\end{proof}

\begin{lemma}
\label{lem:TT4opt-delta6}
Let $k\ge 4$ and set $n=6k+4$.
Let $TT^5_{\mathrm{opt}}(k-1)$ be the family of trees on $6(k-1)+5=6k-1$ vertices
defined in Lemma~\ref{lem:TT5opt-delta6}.
Define $TT^4_{\mathrm{opt}}(k)$ as the family of trees obtained from a tree
$T_5\in TT^5_{\mathrm{opt}}(k-1)$ by choosing an arbitrary edge of $T_5$,
subdividing it by a new vertex $s$, and then attaching four pendant vertices to
$s$.

Then a tree $T\in\mathcal{T}^{(6)}_n$ satisfies $P(T)=30$ if and only if
$T\in TT^4_{\mathrm{opt}}(k)$.
In particular, every tree satisfying the parameter description of
Theorem~\ref{thm:delta6-P30} is isomorphic to a member of $TT^4_{\mathrm{opt}}(k)$.
\end{lemma}

\begin{proof}
Let $k\ge 4$ and $n=6k+4$.
By Theorem~\ref{thm:delta6-P30}, a tree $T\in\mathcal{T}^{(6)}_n$ satisfies
$P(T)=30$ if and only if it is uniquely determined by
\begin{equation}
m_{6,6}=3,\qquad
m_{1,6}=4k+6,\qquad
m_{2,6}=2k-6,
\qquad
n_3=n_4=n_5=0.
\label{eq:TT4-target}
\end{equation}

Take $T_5\in TT^5_{\mathrm{opt}}(k-1)$.
Then $|V(T_5)|=6k-1$ and $P(T_5)=20$, so by Theorem~\ref{thm:delta6-P20} we have
\[
m_{6,6}(T_5)=2,\qquad
m_{1,6}(T_5)=4k+2,\qquad
m_{2,6}(T_5)=2k-6,
\qquad
n_3=n_4=n_5=0.
\]
Choose any edge $e$ of $T_5$, subdivide it by a new vertex $s$, and attach four
pendant vertices to $s$.
Then the order increases by $5$, so the resulting tree has $6k+4=n$ vertices, and
$\deg(s)=6$ implies $\Delta=6$.

Independently of the type of $e$, the subdivision deletes $e$ and creates two
edges incident with $s$, while the four pendant attachments add four new leaves.
In particular, exactly one additional edge of type $(6,6)$ is created in total,
so $m_{6,6}$ increases from $2$ to $3$; moreover, four new edges of type $(1,6)$
are added, while the number of $(2,6)$-edges is unchanged.
Consequently,
\[
m_{6,6}=3,\qquad m_{1,6}=4k+6,\qquad m_{2,6}=2k-6,
\]
and $n_3=n_4=n_5=0$ still holds.
Thus the resulting tree satisfies \eqref{eq:TT4-target}, and by
Theorem~\ref{thm:delta6-P30} it follows that $P(T)=30$.
Hence every tree in $TT^4_{\mathrm{opt}}(k)$ satisfies $P(T)=30$.

Conversely, let $T\in\mathcal{T}^{(6)}_n$ satisfy $P(T)=30$.
Then \eqref{eq:TT4-target} holds, so in particular $m_{6,6}=3$.
Hence the subgraph induced by the degree--$6$ vertices has a vertex $s$ incident
with exactly one $(6,6)$-edge, and the degree constraints force $s$ to be adjacent
to exactly four pendant vertices.
Remove these four pendant neighbors and suppress $s$ (delete $s$ and replace its
two remaining incident edges by a single edge).
Denote the resulting tree by $T^\circ$.
This reduces the order by $5$, so $|V(T^\circ)|=6k-1$, and it decreases $m_{6,6}$
by exactly $1$, yielding $m_{6,6}(T^\circ)=2$ while keeping $m_{2,6}$ unchanged and
reducing $m_{1,6}$ by $4$.
Therefore $T^\circ$ satisfies the parameter description in
Theorem~\ref{thm:delta6-P20} for $n=6k-1$, and hence $P(T^\circ)=20$.
By Lemma~\ref{lem:TT5opt-delta6}, we obtain $T^\circ\in TT^5_{\mathrm{opt}}(k-1)$.

Reversing the suppression shows that $T$ is obtained from $T^\circ$ by subdividing
one edge and attaching four pendant vertices to the new vertex, which is exactly
the defining operation of $TT^4_{\mathrm{opt}}(k)$.
Therefore $T\in TT^4_{\mathrm{opt}}(k)$.
\end{proof}

We state the following lemma, which provides a construction of a family of trees $T\in\mathcal{T}^{(6)}_n$ satisfying the parameter description given in Theorem~\ref{thm:delta6-P40}\textup{(i)}, without proof, since its proof relies on techniques similar to those used in the proofs of the lemmas in this subsection.

\begin{lemma}
\label{lem:TT3iopt-delta6}
Let $k\ge 5$ and set $n=6k+3$.
Let $TT^4_{\mathrm{opt}}(k-1)$ be the family of trees on $6(k-1)+4=6k-2$ vertices
defined in Lemma~\ref{lem:TT4opt-delta6}.
Let $TT^{3}_{(i)}(k)$ denote the family of trees obtained from a tree
$T_4\in TT^4_{\mathrm{opt}}(k-1)$ by choosing an edge $e$ of $T_4$, subdividing $e$
by a new vertex $s$, and then attaching four pendant vertices to $s$.

Then a tree $T\in\mathcal{T}^{(6)}_n$ satisfies the parameter description in
Theorem~\ref{thm:delta6-P40}\textup{(i)} if and only if $T\in TT^{3}_{(i)}(k)$.
In particular, every tree attaining the minimum in
Theorem~\ref{thm:delta6-P40}\textup{(i)} is isomorphic to a member of
$TT^{3}_{(i)}(k)$.
\end{lemma}

\begin{lemma}
\label{lem:TT3iiopt-delta6}
Let $k\ge 2$ and set $n=6k+3$.
Let $TT^1_{\mathrm{opt}}(k)$ be the tree on $6k+1$ vertices defined in
Lemma~\ref{lem:TT1opt-delta6}, with underlying path
$P_{2k+1}=v_1v_2\cdots v_{2k+1}$.

Define $TT^{3}_{(ii)}(k)$ as the family of trees obtained from
$TT^1_{\mathrm{opt}}(k)$ by choosing two odd indices $i<j$ with
$3\le i<j\le 2k-1$ and performing the following operation:

delete the two edges $v_{i-1}v_i$ and $v_{j-1}v_j$, introduce a new vertex $s$,
add the three edges
\[
sv_{i-1},\quad sv_i,\quad sv_j,
\]
and finally attach one new leaf $x$ to $v_{j-1}$.

Then a tree $T\in\mathcal{T}^{(6)}_n$ satisfies the parameter description in
Theorem~\ref{thm:delta6-P40}\textup{(ii)}, namely
\[
m_{2,3}=2,\qquad m_{3,6}=1,\qquad
m_{1,6}=\frac{2n+3}{3}=4k+3,\qquad
m_{2,6}=\frac{n-15}{3}=2k-4,
\]
if and only if $T\in TT^{3}_{(ii)}(k)$.
In particular, every minimizing tree from
Theorem~\ref{thm:delta6-P40}\textup{(ii)} is isomorphic to a member of
$TT^{3}_{(ii)}(k)$.
\end{lemma}

\begin{proof}
Fix $k\ge 2$ and set $n=6k+3$.

Start with $TT^1_{\mathrm{opt}}(k)$.
In this tree all degrees belong to $\{1,2,6\}$ and
\[
m_{1,6}=4k+2,\qquad m_{2,6}=2k-2,
\qquad m_{a,b}=0\ \text{for all other }(a,b)\in S .
\]
Choose odd indices $i<j$ with $3\le i<j\le 2k-1$.
Then $v_i$ and $v_j$ are internal odd vertices of degree $2$, while
$v_{i-1}$ and $v_{j-1}$ are even vertices of degree $6$.

Delete the edges $v_{i-1}v_i$ and $v_{j-1}v_j$.
Introduce a new vertex $s$ and add the edges $sv_{i-1}$, $sv_i$, and $sv_j$.
Finally, attach a new leaf $x$ to $v_{j-1}$.

After this modification the vertex $s$ has degree $3$ and is adjacent to exactly
one vertex of degree $6$ (namely $v_{i-1}$) and to exactly two vertices of degree
$2$ (namely $v_i$ and $v_j$), so $m_{3,6}=1$ and $m_{2,3}=2$ hold.

The two deleted edges remove two edges of type $(2,6)$, while the new edge
$xv_{j-1}$ creates one edge of type $(1,6)$; all remaining new edges are of types
$(2,3)$ or $(3,6)$.
Therefore,
\[
m_{1,6}=(4k+2)+1=4k+3,\qquad
m_{2,6}=(2k-2)-2=2k-4,
\]
and no other $m_{a,b}$ becomes nonzero.
Hence every tree in $TT^{3}_{(ii)}(k)$ satisfies the parameter description in
Theorem~\ref{thm:delta6-P40}\textup{(ii)}.

Conversely, let $T\in\mathcal{T}^{(6)}_n$ satisfy the parameter description in
Theorem~\ref{thm:delta6-P40}\textup{(ii)}.
Since $m_{3,6}=1$ and $m_{2,3}=2$, there is a unique vertex $s$ of degree $3$
adjacent to exactly one vertex $v_{i-1}$ of degree $6$ and to exactly two
vertices $v_i$ and $v_j$ of degree $2$.
Each of $v_i$ and $v_j$ has exactly one further neighbor distinct from $s$; since
no edges of type $(1,2)$, $(2,2)$ or $(3,3)$ occur, these neighbors must be
degree--$6$ vertices, which we denote by $v_{i+1}$ and $v_{j+1}$, respectively.

Moreover, $m_{1,6}=4k+3$ implies that there exists a leaf $x$ adjacent to a
degree--$6$ vertex; necessarily this leaf is adjacent to $v_{j-1}$.

Now delete the leaf $x$ and delete the vertex $s$ together with its three
incident edges, and add back the two edges $v_{i-1}v_i$ and $v_{j-1}v_j$.
Denote the resulting tree by $T^\circ$.
This removes exactly two vertices, so $|V(T^\circ)|=6k+1$.
By construction, all degrees in $T^\circ$ belong to $\{1,2,6\}$ and all edges are
of type $(1,6)$ or $(2,6)$.
Furthermore,
\[
m_{1,6}(T^\circ)=4k+2,\qquad
m_{2,6}(T^\circ)=2k-2,
\qquad
m_{a,b}(T^\circ)=0\ \text{for all other }(a,b)\in S .
\]
Hence $T^\circ$ satisfies the parameter description of
Theorem~\ref{thm:delta6-P0}, and by Lemma~\ref{lem:TT1opt-delta6} we conclude that
\[
T^\circ \cong TT^1_{\mathrm{opt}}(k).
\]

Reversing the above reduction shows that $T$ is obtained from
$TT^1_{\mathrm{opt}}(k)$ by choosing odd indices $i<j$ and performing exactly the
operation in the definition of $TT^{3}_{(ii)}(k)$.
Therefore $T\in TT^{3}_{(ii)}(k)$.
\end{proof}

\section{Conclusion}

In this paper we determined the extremal values of the $\sigma$-irregularity index 
for trees of fixed order $n$ with maximum degree $\Delta=6$, and we completely 
characterized all trees attaining these values. 
We showed that the maximum of $\sigma(T)$ depends rigidly on the residue class of 
$n$ modulo $6$, leading to exactly six regimes with penalty values 
$0,10,20,22,30,$ and $40$. 
For each case we derived explicit degree distributions and edge-type multiplicities, 
and we constructed all extremal trees via concrete recursive operations.

A notable feature of our analysis is the large number of case distinctions 
required even for the moderate bound $\Delta=6$. 
The proofs combine handshake identities, congruence arguments, and sharp lower 
bounds from the penalty table for edge types. 
This already reveals the substantial combinatorial complexity of the problem 
in the tree setting and for relatively small maximum degree.

Together with earlier results for $\Delta=4$ and $\Delta=5$, our findings suggest 
a clear general trend. 
As $\Delta$ increases, both the number of admissible local degree patterns and the 
structural diversity of extremal trees grow rapidly. 
In particular, we expect that for $\Delta \ge 7$ the extremal configurations will 
no longer involve only degrees $\{1,2,\Delta\}$, but will also include additional 
intermediate degrees, as already observed here for $n \equiv 3 \pmod 6$, where 
degree-$3$ vertices occur in one exceptional extremal family.

From a broader perspective, a full classification of $\sigma$-maximizing trees for 
arbitrary $\Delta$ appears highly nontrivial. 
The number of residue classes modulo $\Delta$ increases, the feasible edge types 
widen, and the associated penalty systems become more intricate. 
For large $\Delta$, a brute-force case analysis along the present lines is unlikely 
to be practical.

Nevertheless, our results indicate that a complete characterization may still be 
possible in those congruence classes modulo $\Delta$ for which the minimal penalty 
value remains small. 
This points to a promising research direction: first resolving the 
``low-penalty'' cases and gradually extending the theory to larger $\Delta$. 
We hope that the framework developed here will serve as a useful basis for future 
work on extremal problems for $\sigma$-irregularity beyond $\Delta=6$.


\begin{thebibliography}{99}

\bibitem{Abdo2014}
H.~Abdo, N.~Cohen, and D.~Dimitrov, 
Graphs with maximal irregularity, 
\emph{Filomat}, 28 (2014), 1315--1322.

\bibitem{Abdo2018}
H.~Abdo, D.~Dimitrov, and I.~Gutman, 
Graphs with maximal $\sigma$-irregularity, 
\emph{Discrete Appl. Math.}, 250 (2018), 57--64.

\bibitem{Albertson1997}
M.~O. Albertson, 
The irregularity of a graph, 
\emph{Ars Combin.}, 46 (1997), 219--225.

\bibitem{Arif2023}
S.~Arif, K.~Hayat, and S.~Khan, Spectral bounds for irregularity indices and their applications in QSPR modeling, \emph{J. Appl. Math. Comput.}, 72 (2023), 6351–6373.



\bibitem{Dimitrov2026}
D.~Dimitrov, Ž.~Kovijanić Vukićević, G.~Popivoda, J.~Sedlar, R.~Škrekovski, and S.~Vujošević, 
The $\sigma$-irregularity of trees with maximum degree 5, 
\emph{Discrete Appl. Math.}, 382 (2026), 124--136.



\bibitem{Estrada2010}
E.~Estrada, 
Quantifying network heterogeneity, 
\emph{Phys. Rev. E}, 82 (2010), 066102.

\bibitem{Gutman2005}
I.~Gutman, P.~Hansen, and H.~Mélot, 
Variable neighborhood search for extremal graphs. 10. Comparison of irregularity indices for chemical trees, 
\emph{J. Chem. Inf. Model.}, 45 (2005), 222--230.

\bibitem{Gutman2018}
I.~Gutman, M.~Togan, A.~Yurttaş, A.~S. \c{C}evik, and I.~N. Cangül, 
Inverse problem for sigma index, 
\emph{MATCH Commun. Math. Comput. Chem.}, 79 (2018), 491--508.

\bibitem{Hansen2005}
P.~Hansen and H.~Mélot, 
Variable neighborhood search for extremal graphs. 9. Bounding the irregularity of a graph, 
in: \emph{DIMACS Ser. Discrete Math. Theor. Comput. Sci.}, Vol.~69, 2005, pp.~253--264.

\bibitem{Kovijanic2024}
Ž.~Kovijanić Vukićević, G.~Popivoda, S.~Vujošević, R.~Škrekovski, and D.~Dimitrov, 
The $\sigma$-irregularity of chemical trees, 
\emph{MATCH Commun. Math. Comput. Chem.}, 91 (2024), 267--282.

\bibitem{Reti2018QSPR}
T.~Réti, R.~Sharafdini, Á.~Drégelyi-Kiss, and H.~Haghbin, 
Graph irregularity indices used as molecular descriptors in QSPR studies, 
\emph{MATCH Commun. Math. Comput. Chem.}, 79 (2018), 509--524.

\bibitem{Reti2019}
T.~Réti, 
On some properties of graph irregularity indices with a particular regard to the $\sigma$-index, 
\emph{Appl. Math. Comput.}, 344 (2019), 107--115.

\bibitem{Snijders1981}
T.~A.~B. Snijders, 
The degree variance: an index of graph heterogeneity, 
\emph{Social Networks}, 3 (1981), 163--174.

\end{thebibliography}
\end{document}